\documentclass[a4paper, reqno,12pt]{amsart}

\usepackage{standalone, tikz, tikz-3dplot, appendix}
\usepackage{amsmath, amsthm, amssymb,graphics, mathtools}
\usepackage[retainorgcmds]{IEEEtrantools}
\theoremstyle{plain}
\newtheorem{te}{Theorem}[section]
\newtheorem{lem}[te]{Lemma}
\newtheorem{lemm}{Lemma}

\newtheorem{co}[te]{Corollary}
\newtheorem{pr}[te]{Proposition}

\theoremstyle{remark}
\newtheorem*{re}{Remark}
\newtheorem*{ack*}{Acknowledgment}

\def\y{{\bf y}}\def\ev{{\bf e}}
\def \rt{{\bar{t}}}

\def\t{{\tau}}

\def \k{{\kappa}}

\def\th{{\theta}}

\def \bd{{\bar{\delta}}}
\def \a{{\alpha}}

\def\D{{\Delta}}
\def\bD{{\bar{\Delta}}}

\def\R{{\mathbb R}}

\def\N{{\mathbb N}}
\def\C{{\mathbb C}}

\def\P{{\mathbb P}}

\def\M{{\mathbb M}}
\def \Mb{{\mathbf M}}
\def\A{{\mathbf A}}
\def\L{{\Lambda}}

\def\J {{\mathbf J}}
\def\x{{\mathbf x}}

\def \e{{\epsilon}}

\def\Dec{{\operatorname{\textbf{Dec}}}}
\def\nDec{{\normalfont{\Dec}}}

\def\nint{\mathop{\diagup\kern-13.0pt\int}}

\def\Ic{{\mathcal I}}
\def \Rc{{\mathcal R}} 
 \def \Ic{{\mathcal I}}
\def\Nc{{\mathcal N}}
\def\Jc{{\mathcal J}}
\def\Rc{{\mathcal R}}
\def\Ac{{\mathcal A}}
\def\Pc{{\mathcal P}}
\def\Sc{{\mathcal S}}

\def \Pf{{\mathfrak P}}
\def \Cf{{\mathfrak C}}
\def \Mf{{\mathfrak M}}
\def \Db{{\mathbf D}}
\def \bD{{\bar{\D}}}
\def \Pco{{\overline{\Pc}}}

\def\emph#1{{\it #1}}

\begin{document}

		\author{D\'ominique Kemp}
		\address{School of Mathematics, Institute for Advanced Study,  Princeton NJ}
		\email{dekemp@math.ias.edu}

		\title[Decoupling for ruled hypersurfaces generated by a curve]{Decoupling for ruled hypersurfaces generated by a curve}\maketitle

\begin{abstract}
We extend previous work on the two-dimensional developable tangent surface to its higher dimensional analogues $\Mf \subset \R^{n+1}$. The approach here similarly applies cylindrical approximate decoupling at its core, albeit in a new format. However, the presence of additional rulings as $n$ increases necessitates a case-by-case analysis, which in itself reveals interesting aspects of the geometry of $\Mf$. The contributions of this paper can be viewed as culminating in the optimal $\ell^2(L^p)$ decoupling over Frenet boxes approximating a suitably defined, arbitrarily thin neighborhood of a curve $\phi$.
\end{abstract}
\maketitle

\section{Introduction}
\label{s1}

Let $\phi: [-1,1] \rightarrow \R^{n+1}$ be a nondegenerate, $C^{n+1}$ curve. We consider here the $n$-dimensional (compact) hypersurfaces $\Mf^n \subset \R^{n+1}$ parametrized by \begin{equation} \label{1.1} \y(t, s_1, s_2, \dots, s_{n-1}) = \phi(t) + s_1 \phi'(t) + s_2 \phi''(t) + \cdots + s_{n-1} \phi^{(n-1)}(t), \qquad t \in [-1,1], s_i \in [-2,2].\end{equation} Throughout, we shall use $s$ to abbreviate the tuple $$s\coloneqq (s_1, \dots, s_{n-1}).$$ Expanding negligibly around $\Mf^n$ in a suitable transverse direction, we approximate $\Mf^n$ by a set of positive volume: \begin{equation} \label{nbhd} \Nc_\delta(\Mf) = \{\y(t,s) + v\phi^{(n+1)}(t)/(n+1)! : |v| \in [0, \delta] \}.\end{equation} The normalization chosen for the vector $\phi^{(n+1)}(t)$ may seem curious; its rationale is to be found in Section \ref{lastsec} where it naturally arises. Fourier projections of the form $$P_S f = \int_S e(x \cdot \xi) \hat{f}(\xi) d\xi$$ onto sets $S \subset \Nc_\delta(\Mf)$ will recur throughout. ($e(r) = e^{2\pi i r}$)

The following theorem is the primary objective of this paper.

\begin{te} \label{t0} Let $2 \leq p \leq 6$. For each $\delta > 0$, there exists a partition $\Pc_\delta(\Mf^n)$ of $\Mf^n$ by almost flat subsets $\D$ such that the following $\ell^2(L^p)$ decoupling inequality holds: \begin{equation} \label{1.0} \|P_{\Nc_\delta(\Mf^n)} f\|_p \lesssim_{\phi, n, \e} \delta^{-\e} (\sum_{\Delta \in \Pc_\delta(\Mf^n)} \|P_{\Nc_\delta(\Delta)} f\|_p^2)^{1/2}, \end{equation} valid for all $\e > 0$. The dependence on $\phi$ of the constant in \eqref{1.0} is determined by the minimum value of the (rescaled) Wronskian determinant of $\phi'$ and the $C^{n+2}$ norm of $\phi$. \end{te} 

Inequality \eqref{1.0} is referred to in the literature as an $\ell^2$ {\em decoupling.} The inequality does lead to an $\ell^2$ decoupling by almost rectangular caps, which shall be explained in more detail in the sequel.

We obtain Theorem \ref{t0} by specializing to the {\em moment surface} $\M^n$. These are the hypersurfaces generated by the moment curve $$t \mapsto (t, t^2, \dots, t^{n+1}).$$ They are parametrized by \begin{equation} \label{1.01} \x(t,s) = \sum_{i=1}^{n+1} (t^i + \sum_{j=1}^i (i)_j s_j t^{i-j})\ev_i \end{equation} where $(i)_j$ is the usual notation for the falling factorial $i!/(i-j)!.$

\begin{te} \label{t0.1} Let $2 \leq p \leq 6$, and let $\M^n$ be the $n$-dimensional moment surface. There exists a constant $\k > 0$ such that the following holds. For each $\delta > 0$, there exists a partition $\Pc_\delta(\M^n)$ by almost flat subsets $\D$ such that we have the $\ell^2(L^p)$ decoupling inequality: \begin{equation} \label{1.02} \|P_{\Nc_\delta(\M^n)} f\|_p \leq  (60\delta^{-\e/2})^{\k n/2} n^{1/2} \big(\kappa \log \big(1/\delta \big)\big)^{\k\e^{-1}n \log n} \Big(\sum_{\Delta \in \Pc_\delta(\M^n)} \|P_{\Nc_\delta(\Delta)} f\|_p^2 \Big)^{1/2}, \end{equation} valid for all $\e > 0$.  \end{te} 

Theorem \ref{t0} will follow from Theorem \ref{t0.1} by a straightforward induction on scales.

\begin{re} The caps $\D$ in both theorems as well as the constant in \eqref{1.0} are quantitatively described in this paper. We have saved those assessments for the proofs of the theorems, where they may be found as \eqref{7.14} and Theorem \ref{constant2} respectively. Notably, conventional scenarios do exist where the $n$ dependence of \eqref{1.02} almost entirely vanishes. These are presented in \eqref{7.13} and the remark following, which show that the exponents $n/2$ and $n \log n$ in \eqref{1.02} may be replaced by 1 in their context.  \end{re}

In order to calibrate the dimensions of the caps, it is necessary to obtain a preliminary basic decoupling for $\M$. For this purpose, let us define $\Ic_k = [2^{-k}, 2^{-k+1})$ and decompose $\M^n$ into subsets of the form $$\Ac^\a_{k_1, \dots, k_{n-1}} = \x([-1,1] \times R^{\alpha}_{k_1, \dots, k_{n-1}})$$ where $\alpha: \{1, \dots, n-1\} \rightarrow \{+, -\}$ is a function and \begin{equation} \label{1.2} R^\a_{k_1, \dots, k_{n-1}} = \{(s_1, \dots, s_{n-1}): \alpha(j) s_j \in \Ic_{k_j} \}.\end{equation} Each $k_j$ ranges among the integers inclusively between 0 and the value $\bar{k_j}$ such that \begin{equation} \label{1.20} 2^{-\bar{k_j}+1} \leq (1/j!)\delta^{j/(n+1)} < 2^{-\bar{k_j}+2}. \end{equation} When some $k_j$ equals $\bar{k_j}$, we modify the corresponding half-open interval in \eqref{1.2} to $[0, 2^{-\bar{k_j} +1})$.
The decomposition of $\M^n$ is reflected in the trivial inequality obtained by H\"older's inequality: \begin{equation} \label{beg} \|P_{\Nc_\delta(\M^n)} f\|_p \leq 2^{(n-1)/2} \prod_{j=1}^{n-1}\log \big(j!\delta^{-j/(n+1)} \big) (\sum_\a \sum_{k_1=0}^{\overline{k_1}} \cdots \sum_{k_{n-1}=0}^{\overline{k_{n-1}}} \|P_{\Nc_\delta(\Ac^\a_{k_1, \dots, k_{n-1}})} f\|_p^2)^{1/2}. \end{equation}

Our main task is to confirm Theorem \ref{t0.1}, and we shall find good reason to specialize even further to the four-dimensional setting first. The three-dimensional setting was proven in \cite{K1}, although the argument presented here does apply with less difficulty. In four dimensions, \eqref{1.01} becomes \begin{equation} \label{1.3} \x(t, s) = (t+ s_1, t^2 + 2ts_1 + 2s_2, t^3 + 3t^2s_1 + 6ts_2, t^4 + 4t^3s_1 + 12t^2s_2). \end{equation} The proof of the following decoupling theorem for $\M = \M^3$ gives opportunity for introducing our methods.

\begin{te} \label{t1} Let $\delta > 0$. For each choice of $k_j$, we partition $[-1,1]$ into intervals $I$ of length $\min\{2^{k_1 - k_2}, (2^{k_2}\delta)^{1/2}\}$. Let $\D_{k_1, k_2}$ denote the images $\x(I \times R^\a_{k_1, k_2})$. $\D_{k_1, k_2}$ is almost rectangular.\newline

For all $2 \leq p \leq 6$, $\bigcup_{k_1, k_2} \{\D_{k_1, k_2}\}$ is an $\ell^2(L^p)$ decoupling partition $\Pc_\delta(\M)$.
\end{te}

The proof of Theorems \ref{t0}, \ref{t0.1}, and \ref{t1} will be divulged in sections. We first introduce our primary tool - {\em (parabolic) cylindrical decoupling} - in Section \ref{scyl}.

Throughout the sequel, $p$ will denote a Lebesgue exponent between 2 and 6 inclusive. 

\section{Background}

From the earliest days that Fourier transforms supported on or near Euclidean submanifolds were considered, the parabola has featured prominently. It is well-known that Stein posed the restriction conjecture for the $n$-dimensional paraboloid in the 1960's. Shortly after, H\"ormander produced the proof for $n=1$ \cite{H}. Following the advent of decoupling in Wolff's seminal work on local smoothing \cite{W}, Bourgain and Demeter leveraged the restriction theorem for the parabola with induction on the Euclidean dimension and multilinear methods to prove decoupling for the paraboloid \cite{BD}. That same article also confirms decoupling for the $n$-dimensional cone and arbitrary positively curved hypersurfaces in $\R^{n+1}$. 

Since then, it has been of special interest to the Fourier analysis community to determine the precise decoupling constant for the parabola $\P$. Let us set the stage. Given $\delta > 0$, we take $\Pc_\delta(\P)$ to be a partition of $\P$ by almost flat caps of the form $$\th = I \times \R \cap \P.$$ (See (1) on the first page of \cite{BD}.) That the caps are {\em almost flat} means that $\th$ lies within $\delta$ of any of its tangent lines, i.e. $$\sup_{t_1, t_2 \in I} |\varphi(t_2)- \varphi(t_1) - \varphi'(t_1)(t_2 -t_1)| \leq \delta$$ where $\varphi(t) = t^2$. More generally, we say that any subset (cap) $S$ of a hypersurface is almost flat if for any plane $P$ tangent to $S$ at some point, every point in $S$ lies at most $\delta$ above or below a point in $P$. Such subsets are referred to as ``caps" of the ambient hypersurface because the tangent plane $P$ and a vertical translate of $P$ subtend them.

We define $\Dec^\P_p(\delta)$ to be the smallest $K > 0$ such that the following holds for all $f$ Fourier supported in $\Nc_\delta(\P)$: \begin{equation} \label{pardef} \|f\|_p \leq K (\sum_{\th \in \Pc_\delta(\P)} \|P_{\Nc_\delta(\th)} f\|_p^2)^{1/2}. \end{equation} The monumental contribution of Bourgain and Demeter was to verify that $$\Dec_6^\P(\delta) \lesssim_\e \delta^{-\e}.$$ Z. Li \cite{ZL1} followed their work with carefully tracing that their methods produced the estimate $$\Dec_6^\P(\delta) \lesssim \exp\Big(O \Big(\frac{\log 1/\delta \cdot \log \log \log 1/\delta)}{\log \log 1/\delta}\Big)\Big).$$ Subsequently, taking inspiration from efficient congruencing, he was able to show that $$\Dec_6^\P(\delta) \lesssim \exp \Big(O\Big(\frac{\log 1/\delta}{\log \log 1/\delta}\Big)\Big)$$ \cite{ZL2}. The current best bound on $\Dec_6^\P(\delta)$, and therefore $\Dec_p^\P(\delta)$ by interpolation, is \begin{equation} \label{gmw} \Dec_6^\P(\delta) \lesssim (\log 1/\delta)^c \end{equation} for some absolute constant $c$, obtained by Guth, Maldague, and Wang in \cite{GMW}. As the authors mention there, it has been conjectured that $\Dec_p^{\P}(\delta) \leq C_p$ for $1 \leq p < 6$.

In this paper, it is our aim to contend successfully that the geometry of $\M$, hence $\Mf$, closely relates to that of $\P$. In particular, letting $\Dec_p^\Mf(\delta)$ be defined analogous to \eqref{pardef}, we obtain a bound upon $\Dec_p^\Mf(\delta)$ that is an expression involving $\Dec_p^\P(\delta)$, together with logarithms of $\delta^{-1}$ that cannot be avoided. Let us remark for the sequel that we do know $\Dec_p^\M(\delta)$ to be monotonically decreasing, as is true for $\Dec_p^\P(\delta)$. The proof is by rescaling, just as for the parabola.

Lastly, we explain some of our motivation in seeking Theorem \ref{t0}. $\Mf$ is a zero curvature hypersurface with just one nonvanishing principal curvature, which is essentially $s_{n-1}^{-1}$. Obtaining its decoupling is a natural next step following upon the proof of decoupling over the cone given in \cite{BD}. On the other hand, the ruled hypersurface \eqref{1.1} can be thought of as a ``smoothing" of the Frenet boxes associated with $\phi(t)$ for each $t \in [-1,1].$ After all, the Frenet boxes follow the orientation provided by the Gram-Schmidt orthogonalization of $\phi'(t), \dots, \phi^{(n+1)}(t)$; and the latter appear in \eqref{1.1}. In turn, the Frenet boxes feature naturally in the recent progress on the maximal averages over curves conjecture, proving that conjecture for the 3-dimensional case \cite{BGS, KLO, KLO2, PS}. It is hoped that the decouplings provided here may be useful in the resolution of the conjecture in higher dimensions. However, our analogy relating $\Mf$ to Frenet boxes does not hold if the caps $\D$ in Theorem \ref{t0} do not actually resemble boxes. The key detail required is for the $\D$ to be well approximated by rectangular boxes \footnote{Parallelepipeds work fine too!} contained within them. This geometric trait is conventionally labeled {\em almost rectangular}, to be distinguished from almost flat. In the appendix, we address this matter and show how the decoupling for almost flat caps in Theorem \ref{t0} may be leveraged to obtain the $\ell^2$ decoupling over almost rectangular caps.

\section{Notation and Outline of the Paper}

It is our aim in this paper to carefully trace and limit so far as possible the dependence on $n$ of the decoupling constants for $\M$ and $\Mf$. For this reason, we are careful to stipulate that the customary notation ``$\lesssim$" and ``$\sim$" shall hold no dependence on $n$, in addition to being independent of $\delta, \e, \phi, p$, and of course $f$. 

$$\Phi \lesssim \Upsilon \text{ and } \Phi \sim \Upsilon \  \Longrightarrow \ \Phi \leq A\Upsilon \text{ and } B \Upsilon \leq \Phi \leq A \Upsilon $$ $$\text{ where $A$ and $B$ are constants that are independent of $n, \delta, \e, \phi, p,$ and $f$.} $$\\

The scope of the paper is as follows. In the next section, we start to build the bridge from our problem to parabolic decoupling, showing how the latter may be extended to the $n$-dimensional setting via cylindrical decoupling. At its core, our argument throughout this paper is iterative, as has been customary for proofs of decoupling in the literature. Section \ref{transl} presents the essential ingredient, Lemma \ref{l2}, which makes iterative decoupling possible for $\M$. In Section \ref{s4}, we provide all the rest of the machinery that sustains our iterative decoupling scheme for $\M$. The key achievement of this section is to tie closely the geometry, and hence decoupling, of cylinders over various approximate parabolas to that of $\Nc_\delta(\M)$ at appropriate scales. Section \ref{s4dim} presents the proof of Theorem \ref{t1}, providing occasion for demonstrating the application of our methods in a simple setting. In Section \ref{sgen}, we prove Theorem \ref{t0.1}, and Theorem \ref{t0} is proved following in Section \ref{lastsec}. Finally, the appendix reflects on the question of decoupling over $\Mf$ with almost rectangular caps and paves a path toward its resolution.

\section{Cylindrical decoupling} \label{scyl}

Cylindrical decoupling has become ubiquitous in the literature since the seminal work of Bourgain and Demeter. It is useful because it provides a way of lifting known decoupling inequalities into settings with higher dimensions. Precisely, we state the following.

\begin{pr} \label{pr2} Let $S_i \subset \R^n$ and $S = \bigcup_i S_i$. Assume that the inequality $$\|P_S f\|_p \leq C (\sum_i \|P_{S_i} f\|_p^2)^{1/2}$$ holds for all $f: \R^n \rightarrow \C$. Then, $$\|P_{S \times \R} F\|_p \leq C (\sum_i \|P_{S_i \times \R} F\|_p^2)^{1/2}$$ holds for all complex-valued $F$. \end{pr}

The proof is implied by Fubini's theorem and Minkowski's inequality. \\

Combining the above with the Bourgain-Demeter $\ell^2$ decoupling for the parabola yields the following cornerstone of our subsequent arguments.

\begin{co} \label{co1} Consider the parabolic cylindrical neighborhood \begin{equation} \label{parc} \mathfrak{P}^i = \{(\xi_1, \dots, \xi_{i-1}, \xi, \xi^2 + \eta, \xi_{i+2}, \dots, \xi_{n+1}): |\xi| \lesssim 1, |\eta| \leq E, \xi_j \in \R\}.\end{equation} Partition the $\xi_i$ axis into intervals $I$ of length $E^{1/2}$, and let $$\th =  \R \times \cdots \times \R \times I \times \R \times \cdots \times \R \cap \mathfrak{P^{i}}.$$ We have the following decoupling inequality: \begin{equation} \label{cyl} \|P_{\mathfrak{P}^i} f\|_p \leq \nDec (E ) (\sum_\th \|P_\th f\|_p^2)^{1/2}. \end{equation} \end{co} 

Throughout this paper, we shall like to think collectively of $\eta$ and $E$ as ``error", a measure of how far $\Pf^i$ deviates from the parabolic cylinder mentioned in \eqref{parc}.\\

We immediately obtain partial progress toward Theorem \ref{t0}.

\begin{lem} \label{le1} Let $\delta > 0$. Given $n \geq 3$, partition $[-1,1]$ into intervals $J$ of length $\sim n^{-1}\delta^{1/(n+1)}$ and let $\bar{\D} = \x(J \times R_{\bar{k}_1, \dots, \bar{k}_{n-1}})$. We have $$\|P_{\Nc_\delta(\Ac_{\bar{k}_1, \bar{k}_2, \dots, \bar{k}_{n-1}})} f\|_p \lesssim n^{1/2}\nDec^\P_p (\delta^{2/(n+1)} ) (\sum_{\bar{\D}} \|P_{\Nc_\delta(\bar{\D})} f\|_p^2)^{1/2}.$$ \end{lem}

\begin{proof} Let us abbreviate $\Ac = \Ac_{\bar{k}_1, \dots, \bar{k}_{n-1}}$. Recall that within $\Ac$, $$0 \leq s_j \leq (1/j!) \delta^{j/(n+1)}$$ for each $j$. In particular, $$(t+s_1)^2 - (t^2 + 2ts_1 + 2s_2) = s_1^2 - 2s_2 \leq 2\delta^{2/(n+1)} $$ throughout $\Ac$. In other words, $\Ac$ lies within the parabolic cylindrical neighborhood $\mathfrak{P}^1$ with error $E = 4\delta^{2/(n+1)}$, implying that $$P_{\Nc_\delta(\Ac)} f = P_{\Pf^1} (P_{\Nc_\delta(\Ac)} f).$$ Thus, Corollary \ref{co1} enables us to take a decoupling partition $\Pc_\delta'$ of $\Ac$ that merely dissects the $\xi_1$ axis into segments of length $2\delta^{1/(n+1)}$ : \begin{equation} \label{1.31} \|P_{\Nc_\delta(\Ac)} f\|_p \Dec_p^\P \big(\delta^{2/(n+1)} \big) ( \sum_{\th} \|P_{\th \cap \Nc_\delta(\Ac)} f\|_p^2)^{1/2}. \end{equation} 

Obtaining $\Pc_\delta$ from $\Pc_\delta'$ as stated above occurs from the following observation. Each $\bar{\D}$ is contained in an interval of the form $[b, b + 3\delta^{1/(n+1)}] \times \R^n$ where $b$ shifts by $n^{-1}\delta^{1/(n+1)}$ among consecutive $\bar{\D}$. Therefore, $\bar{\D}$ intersects at most three $\th$, while each $\th$ in turn intersects at most $5n$ $\bar{\D}$. 

Consequently, standard Fourier projection theory applies, since the $\th$ are rectangular: \\ \begin{equation} \label{1.32} \|P_{\th \cap \Nc_\delta(\Ac)} f\|_p \leq (5n)^{1/2} (\sum_{\bD_ \cap \th \ne \emptyset} \|P_{\Nc_\delta(\bD) \cap \th} f\|_p^2)^{1/2} \lesssim_p (5n)^{1/2}(\sum_{\bD \cap \th \ne \emptyset} \|P_{\Nc_\delta(\bD)} f\|^2_p)^{1/2} \end{equation} where the triangle inequality and H\"older's inequality are used initially. Applying \eqref{1.32} to each term on the right side of \eqref{1.31} and counting the possible redundancy of terms confirms $$\|P_{\Nc_\delta(\Ac)} f\|_p \lesssim n^{1/2}\Dec_p^\P(\delta^{2/(n+1)}) (\sum_{\bD \subset \Ac} \|P_{\Nc_\delta(\bD)} f\|_p^2)^{1/2}.$$ 
\end{proof}

\begin{re} In the next section, we prove the existence of a linear map that preserves the standard basis vector $e_{n+1}$ and maps the tangent space of $\M^n$ at some $\x(t,s)$ to the tangent space at $\x(0,s)$. Thus, we see that $\bar{\D}$ in the above theorem is flat since $$|t| \leq n^{-1}\delta^{1/(n+1)}; |s_j| \leq (1/j!) \delta^{j/(n+1)} \quad \Longrightarrow   \quad |t^{n+1} + \sum_{j=1}^{n-1} s_j (n+1)_j t^{n+1 - j}| \leq $$ $$\delta \sum_{j=2}^{n+1} \binom{n+1}{j} (1/n)^j \sim \delta.$$
The reader will recognize the sum in the last line as $$(1+1/n)^{n+1} - 1 - (n+1)/n,$$ so the equivalence follows by an application of Taylor's theorem and the identity $e = \lim_{ n} (1 + 1/n)^n$.
\end{re}

\section{Translation invariance of $L^p$ norm} \label{transl}

The construction of the decoupling partition for the remaining $\Ac^\a_{k_1, \dots, k_{n-1}}$ will involve many scale-based iterations of cylindrical decoupling. We shall want to treat the lengths of each successive decoupling cap $\D^{(j)}$ thus obtained as self-improving. For this, it will be key to have some mechanism for redescribing the manifold $\M$ from the perspective of $\D^{(j)}$. Our chosen device is a linear map that rearranges the first $n+1$ derivatives of $\phi$ as the standard basis of $\R^{n+1}$ with a slight rescaling of each component.

For this section, we specialize $\phi$ to be the moment curve $$\phi(t) = (t, t^2, \dots, t^{n+1}).$$ As well, we shall modify the parametrization $\x$ so that $\phi$ carries a linear coefficient variable also: $$\x_{ext}(t,s_0, s_1, \dots, s_{n-1}) = s_0 \phi(t) + s_1 \phi'(t) + \cdots + s_{n-1}\phi^{(n-1)}(t).$$ Fixing $s_0 = 1$ returns the original parametrization $\x$. In fact, note that $$s_0^{-1}\x_{ext}(t, s_0, s_1, \dots, s_{n-1}) = \x(t, s_1/s_0, \dots, s_{n-1}/s_0).$$ Thus, we clarify that the modification $\x_{ext}$ is made only for technical reasons. It facilitates a broader form of Lemma \ref{l2} below that will be useful in later proofs. 

Fix $t_0, s_0 \in [-1,1]$. Define $\A: \R^{n+1} \rightarrow \R^{n+1}$ by \begin{equation} \label{defA} \A(\xi) = s_0\phi(t_0) + \sum_{j=1}^{n+1} \frac{\phi^{(j)}(t_0)}{j!}\xi_j.\end{equation} Observe that the matrix $\L$ whose columns are given by $(1/j!)\phi^{(j)}(t_0)$ is lower triangular with ones along the diagonal. Thus, $\det(\L) = 1.$

\begin{lem} \label{l2} \begin{equation} \label{2.1} \x_{ext}(t+t_0, s_0, \dots, s_{n-1}) = \A(\x_{ext}(t, s_0, \dots, s_{n-1})).\end{equation} As well, the function $f_{t_0}$ given by $$\hat{f_{t_0}} = \hat{f} \circ \A$$ satisfies \begin{equation} \label{2.2} \|f_{t_0}\|_p = \|f\|_p \end{equation} for all $p $. \end{lem}

Note that if $f$ is Fourier supported in some neighborhood $\Nc_\delta(\D)$ of a cap $\D = \x_{ext}([t_0, t_1] \times \{\bar{s}\} \times J)$, $\bar{s}$ any fixed positive number, then (by change of variables on the frequency side) the Fourier support of $f_{t_0}$ is the $\delta$-neighborhood of the cap $\x_{ext}([0, t_1 - t_0] \times \{\bar{s}\} \times J)$. It is in this sense that we think of the action of $\A$ as translation in the $t$-coordinate. We shall refer to the translation $f_{t_0}$, and the corresponding movement of $\D$, as a relocation to {\em starting position}. 

\begin{proof} 
Note that \eqref{2.2} follows from change of variables combined with the comment following \eqref{defA}. To verify \eqref{2.1}, we shall show that the equality holds for corresponding components. 

Consider the $i$-th component of $\x(t+t_0, s_0, \dots, s_{n-1})$: $$(t+t_0)^is_0 + i(t+t_0)^{i-1}s_1 + \dots + i!(t+t_0)s_{i-1} + i!s_i = \sum_{j=0}^i \frac{i!}{(i-j)!} (t+t_0)^{i-j}s_j,$$ where $s_n = 0 = s_{n+1}.$ Expanding the binomials, we rewrite each summation term as $$\frac{i!}{(i-j)!} s_j\sum_{l=0}^{i-j} {i-j \choose l} t^{i-j-l}t_0^l.$$ Then, we group together the terms that have the same power of $t_0$. Their sum is \begin{equation} \label{2.3} t_0^l(\sum_{j=0}^{i-l} \frac{i!}{(i-j)!} {i-j \choose l} t^{i-j-l}s_j).\end{equation} Simplifying the constant factor in \eqref{2.3}, we may rewrite the expression as \begin{equation} \label{2.4} (\frac{i!}{l!})t_0^l(\sum_{j=0}^{i-l}\frac{1}{(i-j-l)!}t^{i-j-l}s_j) \end{equation} which is exactly the $i$-th component of $((i-l)!)^{-1}\phi^{(i-l)}(t_0)$ multiplied by the $(i-l)$-th component of $\x(t, s_0, \dots, s_{n-1}).$ The summation of \eqref{2.4} over all integral $0 \leq l \leq i$ is then the sum of all nonzero $i$-th components of the terms in \eqref{defA} ($\x$ being precomposed by $\xi_l$). The proof is complete.

\end{proof}

\section{Intermediate decouplings} \label{s4}

In this section, we develop terminology that will aid us in later exposition. The goal is to develop a theory around parabolic cylindrical decoupling that enables us to iteratively decouple unto increasingly smaller caps. This is the core ingredient of the proof of Theorem \ref{t1}.  However, some work must be done to make Corollary \ref{co1} relevant to the decoupling theory of $\M^n$. First, we need to obtain control over the error term $E$, one that renders $E$ scale-dependent. In this way, using Lemma \ref{l2}, diminishing the size of $t$ via prior decouplings would yield increasingly smaller $E$, allowing us to reduce $t$  to the desired smallest scale. This iterative engine may stall though if we are not careful. Corollary \ref{co1} provides a decoupling partition by sets that are rectilinear, not the curved caps that comprise $\M^n$ which are desired. Securing the latter, as we have seen before, is related to the overlap that occurs between the rectilinear and curved sets.  Thus, we arrange a setting where the overlap is minimal, so that an appropriate reformulation of Corollary \ref{co1} holds. \\

Let $S \subseteq R^\alpha_{k_1, \dots, k_{n-1}}$, and consider the corresponding cap \begin{equation} \label{4.00} \D = \x([0, T] \times S), \end{equation} $T \leq 1$. Given an appropriate partition by subintervals $\{J_j\}_{j \in \Jc}$ of $[0, T]$, we aim for a decoupling that partitions $\D$ into smaller sets $\D_j$ \begin{equation} \label{4.00a} \D_j = \x(J_j \times S).\end{equation} We call $J_j$ the {\em $t$-interval} of $\D_j$, and its length the $t$-{\em length} of $\D_j$. 

We shall be concise with notation and use $\D$ also to represent caps mapped by the extended parametrization \begin{equation} \label{4.00e} \D = \x_{ext}([0,T] \times \{a\} \times S).\end{equation} Throughout the sequel, $a$ will always have magnitude less than 1, and thus the proofs that follow will indeed hold for \eqref{4.00} and \eqref{4.00e} alike.

In general, by direct computation, $\D \subset \Pf^i(t)$ where $\Pf^i(t)$ is defined as in Corollary \ref{co1} but with $\eta$ now a function of $t$: \begin{equation} \label{4.0} \eta(t) = (\xi_{i+1} \circ \x)(t,s) - (\xi_i \circ \x)^2(t,s) = O(C_i.t^m) + O(D_i\Lambda) \end{equation} for fixed values $(m, \Lambda) \in \N \times [\delta, 1)$ and some constants $C_i, D_i$ dependent on $i$. The constant $\Lambda$ should be thought of as the error value desired, while the $O(C_i t^m)$ term is the discrepancy which we shall try to reduce via iterative decoupling. The exact form of \eqref{4.0} of course involves rather large coefficients of each power of $t$, and both their number and magnitude grow with $n$. By use of a rescaling argument, we mitigate the impact of this growth in the next subsection. 

We shall say more about \eqref{4.0} in Subsection \ref{ss4.1}, where in particular $m$ shall be determined to be 2. In the meantime, we prepare for the exposition there by adapting Corollary \ref{co1} to $\M^n$.

\subsection{The inductive step for parabolic cylindrical decoupling on $\M^n$} \label{ss4.0} We aim to develop the decoupling for $\M^n$ using that for parabolic cylinders $\Pf^i$. This works directly when $i=1$. In fact, the $t$ and $\xi_1$ coordinates span similarly sized intervals, so long as the scale exceeds the maximum size of $s_1$. Thus, an entirely natural extension of the argument proving Lemma \ref{le1} yields

\begin{lem} \label{lem4.0} Assume $\D \subset \M^n$ is contained in $\Pf^1(E)$ for $E \in \big[\max_\D |s_1|,1\big]$. Then,  \begin{equation} \label{4.111'} \|P_{\Nc_\delta(\D)} f\|_p \lesssim  \nDec_p^\P (E ) (\sum_{\D' \subset \D} \|P_{\Nc_\delta(\D')} f\|_p^2)^{1/2} \end{equation} where $\D$ is partitioned into caps $\D'$ of $t$-length $E^{1/2}$. \end{lem}

In the general setting, the rectilinear caps of Corollary \ref{co1} do not adhere to the geometry of $\M^n$. However, we may artificially fix this if we decouple with small jumps between successive scales. We draw our inspiration here from the proof of decoupling for the cone given in Section 12.2 of \cite{C}.

The jump between scales will be a factor of $\delta^{\e/2}$. Let us point out to the reader that we may always partition the values of $t$ or $s_i$ into subintervals of length $\delta^\e$. Only using the triangle inequality and H\"older's inequality, we incur just $\delta^\e$ loss in so doing: \begin{equation} \label{4.0a} \|P_{\D} f\|_p \leq (\delta^{-\e/2})^n (\sum_{\D' \subset \D} \|P_{\D'} f\|_p^2)^{1/2}.\end{equation} \\ 

Fix $i \geq 2$. In the $n$-dimensional setting, we need a standard formula for the $i$-th component of points in $\Delta$, and so we restrict $t$ and the relevant coefficient $s_{i-1}$ correspondingly. Let $\Ic \subset [(4(i-1)!)^{-1}, (i-1)!^{-1}]$ be an interval of length $\delta^\e/(30(i-1)!)$ containing some number $c_i$ to be determined in the next subsection. Assume that $\D$ satisfies \\ \begin{equation} \label{4.1} \left\{\begin{array}{rl} |s_{i-1}| \in \Ic \\ \\ |s_j| \leq \frac{2}{j!},&  \quad j = 1, \dots, i-2 \\ \\  T \leq \frac{\delta^\e}{30(i+1)} \end{array} \right. \end{equation} \\ for all $s \in S$. We have justified the partitioning by $\delta^\e$ in \eqref{4.0a}, and the size bounds on $s_i$ shall be arranged in Section \ref{sgen} by a rescaling argument. We shall refer to the system of inequalities comprising \eqref{4.1} as condition ($\star$) and to caps satisfying condition ($\star$) as {\em standardized} caps. The second inequality in \eqref{4.1} pertains properly to the $n$-dimensional setting, where we might otherwise accumulate factors of order $n!$ during the decouplings outlined in Section \ref{sgen}. For low-dimensional moment surfaces, $\Ic$ may be slightly larger (of length $\sim \delta^\e                                                                                                                                                                                                                                                                                                                                                                                                                                                                                                                                                                                                                                                                                                                                                                                         $) and the inequalities \begin{equation} \left\{\begin{array}{rl} |s_{i-1}| \in \Ic \\ \\ \max \limits_{j \leq i-2} |s_j| \lesssim 1\\ \\  T \leq \delta^\e \end{array} \right.\end{equation} are taken instead. Such will characterize standardized caps in low dimensions. 

The inequalities specified in condition ($\star$) enable the following reduction.

\begin{lem} \label{lem4.1} Let $m \leq M \leq i + 1$ be positive integers. Suppose that \\ \begin{itemize} \item $|s_j| \leq 1/j!$ \qquad \; for $j = 1, \dots, m$ \\ \item $|t| \leq (2(i+1))^{-1}$. \\ \end{itemize} Then, \begin{equation} \label{4.1a} |t^M + \sum_{j=1}^{m} (M)_j s_j t^{M-j}| \leq 2(M|t|)^{M-m}. \end{equation} \end{lem}

\begin{proof} 

Note that $|s_j| \leq 1/(j!)$ implies that $$|(M)_j s_j| \leq \binom{M}{j}$$ for all $j = 1, \dots, m$. For our purposes, this fact is most useful for the terms in $$\sum_{j=1}^{m} (M)_j s_j t^{M-j}$$ with $j > M - j$. It allows us to compare the sum to a geometric series: $$|t^M| + \sum_{j=1}^{m} (M)_j |s_j t^{M-j}| \; \leq \sum_{j=0}^{m} \binom{M}{j} |t^{M - j}|  \; \leq \sum_{j=0}^{m} (M|t|)^{M-j} \; \leq 2(M|t|)^{M-m}. \quad \qedhere$$ 

\end{proof}

Condition ($\star$) encodes critical geometric information.  It guarantees that at reduction scale $\delta^\e$, the caps $\th^j$ successively provided by Corollary \ref{co1} have bounded overlap with similarly sized caps $\D^j$ described as in \eqref{4.00}. 

\begin{lem} \label{lem4.2} Let $\D$ be in starting position. Assume $\D$ is contained in a set $\Pf^i = \Pf^i(E)$ with $E = i^2\delta^\e T^2$, i.e. \begin{equation} \label{4.1b} \D \subset \{\xi \in \R^{n+1}: \   |\xi_i| \lesssim 1, \; |\xi_{i+1} - \xi_i^2| \leq \delta^\e(iT)^2\}. \end{equation} There exists a fixed constant $\k$ (independent of $n, \epsilon$, and $\delta$) such that if $\D$ satisfies condition ($\star$), then \begin{equation} \label{4.111} \|P_{\Nc_\delta(\D)} f\|_p \leq \k \nDec_p^\P (E ) (\sum_{\D' \subset \D} \|P_{\Nc_\delta(\D')} f\|_p^2)^{1/2} \end{equation} where $\D$ is partitioned into caps $\D'$ of $t$-length $i^{-1}E^{1/2}$. \end{lem}

\begin{proof} A priori, Corollary \ref{co1} provides \begin{equation} \label{4.112} \|P_{\Nc_\delta(\D)} f \|_p \leq \Dec_p^\P(E) (\sum_{\th \cap \Nc_\delta(\Delta) \ne \emptyset} \|P_\th f\|_p^2)^{1/2}. \end{equation} Denoting a partition of the $\xi_i$ axis by intervals $I$ of length $E^{1/2} = i\delta^{\e/2}T$, the sets $\th$ have the form $$\th = \R \times \cdots \times \R \times I \times \R \times \cdots \times \R \cap \Pf^i$$ where $I$ is the restriction of the $i$-th component.  It remains to replace the $\th$ in \eqref{4.112} with the caps $\D'$. To simplify the writing, we shall refer to $\R \times \cdots \times \R \times I \times \R \times \cdots \times \R$ also as an interval with length that of $I$. (Whether in the sequel, we mean by $I$ the rectangle $\R \times \cdots \times \R \times I \times \R \times \cdots \times \R$ or the interval contained in $\R$ will be clear from the context.) $J_{\Delta'}$ will denote the $t$-interval associated with $\D'$ in the sense of \eqref{4.00a}.

If $\D \subset \Pf^1(E)$, the argument given in the proof of Lemma \ref{le1} naturally extends here \big(see \eqref{4.114}\big). In fact, inspired by that proof, our strategy for the case $i \geq 2$ is to find intervals $I_{\Delta'} \supset \D'$ of uniform length $10i\delta^{\e/2}T$ adhering to the following monotonicity condition: \begin{equation} \label{4.1c} I_{\Delta'_2} =  I_{\Delta'_1} + ic(\min J_{\Delta'_2} - \min J_{\D'_1}) \qquad \text{ for all } \D'_1, \D'_2 \in \{\D'\}; \  |c| \in [1/2,1] \text{ fixed}. \end{equation} The intervals $\{I_{\D'}\}$ have length 10 times that of $\th$, and so \eqref{4.1c} implies that each $\th$ intersects at most 25 $\D'$. In turn, each $\D'$ certainly intersects at most 25 $\th$ since this exceeds the maximum number required to cover $I_{\D'}$. Thus, once the $I_{\D'}$ are constructed, \eqref{4.111} follows from \eqref{4.112} as in the end of the proof of Lemma \ref{le1}.

Recall that the $i$-th component of $\x(t,s)$ is \begin{equation} \label{4.1d} t^i + \sum_{j=1}^i (i)_j s_j t^{i-j} = t^i + i s_1 t^{i-1} + \cdots + (i!/2)s_{i-2}t^2 + i!s_{i-1}t + i!s_i. \end{equation} Fixing $t=0$ and $s_{i+1} = 2^{-k_{i+1}}$ in \eqref{4.1b} shows that $i!s_i$ throughout $\Delta$ is $O(E^{1/2})$: \begin{equation} \label{4.113}  (i! 2^{-k_i})^2 \leq \max_{c \; \in \; [(i!2^{-k_i})^2, \  (i! 2^{-k_i + 1})^2]} |c - (i+1)!2^{-k_{i+1}}| \leq E. \end{equation} It follows that \begin{equation} \label{4.114} |i!s_i| \leq 2E^{1/2}. \end{equation} By hypothesis, it is clear that \begin{equation} \label{4.115} i!|s_{i-1} - \bar{s}_{i-1}| t \leq i\delta^\e T < E^{1/2}\end{equation} for all $|t| \leq T; s_{i-1}, \bar{s}_{i-1} \in I$. Finally,
according to Lemma \ref{lem4.1}, \begin{equation} \label{4.116} |t^i + \sum_{j=1}^{i-2} (i)_j s_j t^{i-j}| \leq 4(it)^2 \leq i\delta^{\e}T < E^{1/2} \end{equation} for all $t, s_j$ as in \eqref{4.1}.

We define the intervals $I_{\D'} \subset \R$ as follows. Given $J_{\D'} = [b, b + \delta^{\e/2}T]$, we set\\ \begin{equation} \label{4.117} I_{\D'} = [i!\a_{i-1}c_i b - 5E^{1/2}, i!\a_{i-1}c_i b + 5E^{1/2}]. \end{equation} (If $i = 1$, we just define $I_{\D'}$ as in the proof of Lemma \ref{le1}: $I_{\D'} = [b, b+ 3E^{1/2}].$) 

Let $\x(t,s) \in \D'$. By \eqref{4.114}, \eqref{4.115}, and \eqref{4.116}, $$|(t^i + \sum_{j=1}^i (i)_j s_j t^{i-j}) - i!\a_{i-1}c_i b| \leq 3E^{1/2} + i!|s_{i-1}t - \a_{i-1}c_i b| $$ $$\leq 3E^{1/2} + i!|(s_{i-1} - \a_{i-1}c_i)t| + i!c_i |t-b|$$ $$\leq 4E^{1/2} + i\delta^{\e/2}T = 5E^{1/2} $$ \\ since $c_i \in [(2(i -1)!)^{-1}, (i-1)!^{-1}].$ So the infinite cylinder $I_{\D'} \subset \R^{n+1}$ over \eqref{4.117} contains $\D'$.

\end{proof}

\begin{re} We only proved Lemma \ref{lem4.2} in the ``high-dimensional" context where condition ($\star$) restricts the $s_j$ variables also and not just $s_{i-1}, t$. But in low dimensions, we can obtain the critical inequalities \eqref{4.114}, \eqref{4.115}, and \eqref{4.116} more easily, assuming that $``\leq"$ is weakened to $``\lesssim"$ in each. Indeed, the addition of a dimension-dependent constant factor in each inequality is immaterial when the dimension is kept low. \end{re}

\subsection{Parabolic cylindrical decoupling on standardized caps} \label{ss4.1} Let us revisit \eqref{4.0} and give a proper assessment for a standardized cap $\D$. For all $\x(t,s) \in \D$, \begin{equation} \label{4.118} (\xi_{i+1} \circ \x)(t,s) - (\xi_i \circ \x)^2(t,s) = (t^{i+1} + \sum_{j=1}^{i+1} (i+1)_j s_j t^{i+1- j}) - (t^i + \sum_{j=1}^i (i)_j s_j t^{i-j})^2. \end{equation} where we take $s_n = 0=s_{n+1}$ in order to include the cases $i=n-1$ or $i=n$. We compute the coefficient of $t^2$ in \eqref{4.118} exactly: \begin{equation} \label{4.119} ((i+1)_{i-1} - (i)_{i-1}^2s_{i-1})s_{i-1}t^2 = (\frac{(i+1)!}{2(i!)^2} - s_{i-1})(i!)^2 s_{i-1}t^2. \end{equation} If we choose \begin{equation} \label{4.1190} c_i = \frac{(i+1)!}{2(i!)^2} \end{equation} \\ to be the number promised in the context of \eqref{4.1}, then \eqref{4.119} is bounded above by \begin{equation} \label{4.120} (1/30)\delta^\e i^2t^2. \end{equation} 

Concerning all the other terms in \eqref{4.118}, we may apply Lemma \ref{lem4.1} to the sums over $j < i-1$ and retain the rest: $$|\xi_{i+1} \circ \x - (\xi_i \circ \x)^2| \leq 5((i+1)|t|)^3 + (1/30)i^2 \delta^\e t^2 + (i+1)!|s_{i+1}|+ 4(it)^2(|s_{i-1}t| +|s_i|)i!$$ \begin{equation} \label{4.121}  + i!|s_i t|(i+1 + 2i! |s_{i-1}|) + (i!s_i)^2. \end{equation} Using \eqref{4.1}, we can bound $$ 5((i+1)|t|)^3 + 4(it)^2|s_{i-1}t|i! \leq 9(iT)^2(\delta^\e/30).$$ On the other hand, $$i+1 + 2i!|s_{i-1}| \leq 4i.$$ Therefore, \eqref{4.121} is bounded above by \begin{equation} \label{4.122} (9/30)\delta^\e(iT)^2 + 8i|i!s_i |T + (i+1)!|s_{i+1}| + (i!s_i)^2 \end{equation} for large $i \ne n$. If $i=n$, we have shown that \eqref{4.121} is bounded by $\delta^\e (iT)^2$ and so the $\delta$-neighborhood of $\D$ is contained in $\Pf^n(E)$ with \begin{equation} \label{4.12200} E = (2/3) \delta^\e (iT)^2 + \delta. \end{equation}

From our standpoint, the last two terms in \eqref{4.122} are inevitable; most significantly, they cannot be reduced by any $\ell^2$ decoupling. Therefore, these must be the target $t$-lengths to be attained. Thus inspired, we observe that $$T \geq (30/i)(i!|s_i|) \delta^{-\e} \Longrightarrow i|i!s_i| T \leq (1/30)\delta^\e (iT)^2.$$ Having $T \leq 30(i-1)!|s_i|\delta^{-\e}$ is favorable since H\"older's inequality applies here with minimal loss. So we assume that $T$ is larger, in which case \eqref{4.122} and thus \eqref{4.121} are bounded by \begin{equation} \label{4.1220} (2/3)\delta^\e(iT)^2 + (i+1)!2^{1-k_{i+1} } + (i!2^{1-k_i})^2. \end{equation}

\begin{re} For low-dimensional moment surfaces, the above demonstration is even easier. Taking the trivial bounds $|s_j| \leq 2$, we may bound the sum of all terms of order higher than 1 in \eqref{4.118} by $C_i \delta^\e T^2$. The linear terms altogether are bounded by some $D_i |s_i|T$, and we may reason as in the previous paragraph that $$D_i |s_i|T \leq D_i \delta^\e T^2$$ if T is not (almost) at the desired scale. Therefore, we have the analogue of \eqref{4.1220}: there exist constants $C_1, C_2, C_3$ such that \eqref{4.118} is bounded by $$ C_1 \delta^\e T^2 + C_2 |2^{-k_{i+1}}| + C_3 |2^{-k_i}|^2$$ for all small values of $i$. \end{re}

\begin{pr} \label{pr4.1} Assume that $\Nc_\delta(\D)$ is a cap neighborhood of $t$-length $T$ that is standardized relative to the $i$-th component. Let $L$ be determined by \begin{IEEEeqnarray*}{rCls} L &\geq& i^{-1} \max \{((i+1)!2^{-k_{i+1}})^{1/2}, i!2^{-k_i}\} &\qquad if $1 < i < n-1$  \\  L &=&  (n-2)!2^{-k_{n-1}} & \qquad if $i=n-1$\\ L &=& (1/\sqrt{6})n^{-1}\delta^{1/2} & \qquad  if $i=n$.  \end{IEEEeqnarray*} Then $\Nc_\delta(\D) \subset \Pf^i(E)$ where \begin{equation} \label{4.123} E = (2/3)\delta^\e (iT)^2 + 6(iL)^2,\end{equation} provided that $T \geq 30 \delta^{-\e} L$. In all cases, $\D$ may be partitioned for $\ell^2$ decoupling into caps $\D'$ of $t$-length $ L$: \begin{equation} \label{4.1241} \|P_{\Nc_\delta(\D)} f\|_p \leq  (30\delta^{-\e})^{1/2} \big( \k \nDec_p^\P \big( 6i^2L^2 \big) \big)^{2\e^{-1}\log_\delta{L}} \bigg(\sum_{\D' \subset \D} \|P_{\Nc_\delta(\D')} f\|_p^2 \bigg)^{1/2}. \end{equation}  \\

\end{pr} 

Note that inequality \eqref{4.1241} is a true decoupling whenever $L$ exceeds some power of $\delta$, since the decoupling constant is monotonic.

\begin{proof} \eqref{4.123} has already been verified in the preceding exposition. It remains to deduce the $\ell^2$ decoupling when $T$ exceeds $\delta^{-\e}L$. The following argument is an iterative decoupling procedure that utilizes Lemma \ref{lem4.2} at each step. Since $\D$ is standardized, all subsets of $\D$ are standardized as well, so the requirement of Lemma \ref{lem4.2} is maintained throughout. To ease the notation a bit, we shall refer to $\Dec_p^\P(\cdot)$ only as $\Dec(\cdot)$ beginning with this proof and continuing until Section \ref{lastsec}.

We produce the $\ell^2$ decoupling. Initially, $$\Nc_\delta(\D) \subset \Pf^i \big (\delta^\e (iT)^2 \big)$$ by \eqref{4.123}. Applying Lemma \ref{lem4.2} with $E = i^2\delta^\e T^2$, we obtain for the partition of $\D$ by caps $\D^{(1)}$ of $t$-length $\delta^{\e/2}T$: \begin{equation} \label{4.124} \|P_{\Nc_\delta(\D)} f\|_p \leq  \k \Dec \big(\delta^\e (iT)^2 \big) \big(\sum_{\D^{(1)} \subset \D} \|P_{\Nc_\delta(\D^{(1)})} f\|_p^2 \big)^{1/2}. \end{equation}

Using Lemma \ref{l2}, we translate each $\D^{(1)}$ to starting position, i.e. $t \in [0, \delta^{\e/2}T]$, and deal with each cap individually. Over $\D^{(1)}$, \eqref{4.123} is bounded by $i^2\delta^{\e}(\delta^{\e/2}T)^2$, (provided $\delta^{\e/2}T \geq 30\delta^{-\e}L$), so Lemma \ref{lem4.2} provides \begin{equation} \label{4.125} \|P_{\Nc_\delta(\D^{(1)})} f\|_p \leq  \k \Dec \big(\delta^{2\e}(iT)^2 \big) \big(\sum_{\D^{(2)} \subset \D^{(1)}} \|P_{\Nc_\delta(\D^{(2)})} f\|_p^2 \big)^{1/2}, \end{equation} each $\D^{(2)}$ having $t$-length $\delta^\e T$. The juxtaposition of \eqref{4.124} and \eqref{4.125} provides $$ \|P_{\Nc_\delta(\D)} f\|_p \leq \k^2 \Dec \big(\delta^\e (iT)^2 \big) \Dec \big(\delta^{2\e}(iT)^2 \big)\big(\sum_{\D^{(2)} \subset \D} \|P_{\Nc_\delta(\D^{(2)})} f\|_p^2 \big)^{1/2}.$$

We repeat. After $J$ iterations, we have \begin{equation} \label{4.12} \|P_{\Nc_\delta(\D)} f\|_p \leq  \k^J \prod_{j=1}^{J}\Dec \big(\delta^{j\e}(iT)^2 \big)  \big(\sum_{\D^{(J)} \subset \D} \|P_{\Nc_\delta(\D^{(J)})} f\|_p^2 \big)^{1/2}. \end{equation} $\D^{(J)}$ has $t$-length $\delta^{J\e/2}T.$ Since $T \leq 1$ and $2^{1-k_j } \geq \delta^{j/(n+1)}$, $$\delta^{J\e/2} \leq \delta^{-\e} i^{-1} \big[(i+1)!\delta^{(i+1)/(n+1)}\big]^{1/2} \quad \Longrightarrow  \quad \delta^{J\e/2}T < 2\delta^{-\e}L$$ and similarly, when $i=n$, $$\delta^{J\e/2} \leq \delta^{-\e} n^{-1} \delta^{1/2} \quad \Longrightarrow  \quad \delta^{J\e/2}T < 3\delta^{-\e}L.$$Thus, choosing $J$ to be the minimal integer $j$ satisfying $$\delta^{j\e/2}T <  30\delta^{-\e}L,$$ we see that $J$ is bounded: \begin{IEEEeqnarray*}{rCls} J &\leq& \frac{i+1}{\e(n+1)} &\qquad if $i \leq n-1$ \\ \\J &\leq& \frac{1 + O(\log{n}/|\log{\delta}|)}{\e} &\qquad if $i =n$.\end{IEEEeqnarray*} For each $j \leq J$, $$\delta^{j\e}(iT)^2 \geq 900\delta^{-\e}(iL)^2 > 6(iL)^2 \geq \delta^{(i+1)/(n+1)},$$ so by monotonicity of the decoupling constant \begin{equation} \label{4.13} \Dec \big(\delta^{j\e}(iT)^2 \big)  \leq \Dec \big(\delta^{(i+1)/(n+1)} \big). \end{equation} Each $\D^{(J)}$ has length less than $30\delta^{-\e}L$, so we conclude with \eqref{4.1241} by applying H\"older's inequality and \eqref{4.13} to \eqref{4.12}. 

\end{proof}

\begin{lem} \label{lemf}Each translated $\D'$ obtained in Proposition \ref{pr4.1} is flat relative to some $m$-th coordinate hyperplane $\xi_m = 0$  $\big(m \in \{i, i+1\}\big)$ in the following strong sense: \begin{IEEEeqnarray}{rCll}  \big(t^{i+1} + \sum_{j=1}^{i+1} (i+1)_j |s_j|t^{i+1-j}\big) \Big |_{\D} &\lesssim& (iL)^2 & \quad \text{ if } i < n-1 \text{ and } iL = \big((i+1)!2^{-k_{i+1}}\big)^{1/2}, \text{ or } i=n \nonumber \\ & & & \qquad (s_n = 0 = s_{n+1}) \label{iflat1}  \\ \big(t^i + \sum_{j=1}^i (i)_j |s_j|t^{i-j}\big) \Big |_{\D} &\lesssim& iL & \quad \text{ if }i < n-1 \text{ and } iL = i!2^{-k_{i}}, \text{ or } i=n-1 \label{iflat2}. \end{IEEEeqnarray} \end{lem}

\begin{proof} In the setting of \eqref{iflat1}, throughout $\D'$ $$\max \big\{(i+1)! |s_i| t, (i+1)! |s_{i+1}| \big \} \lesssim (iL)^2$$ by hypothesis. Thus, \eqref{iflat1} is a matter of verifying that $$t^{i+1} + \sum_{j=1}^{i-1} (i+1)_j |s_j|t^{i+1-j} \lesssim (iL)^2.$$ Since $\D'$ is standardized, the sum is bounded by $$ \sum_{j=0}^{i-1} \binom{i+1}{j} t^{i+1-j} = (1+t)^{i+1} - 1-(i+1)t $$ which in turn by Taylor approximation is bounded by $$(i+1)i(1 + 1/i)^{i-1}L^2 \lesssim (iL)^2.$$

Similarly, \eqref{iflat2} follows from $$t^i + \sum_{j=1}^{i-1} (i)_j |s_j| t^{i-j} \leq \sum_{j=0}^{i-1} \binom{i}{j} t^{i-j} = (1+t)^i - 1 \leq i(1+1/i)^{i-1}L \lesssim iL.$$

\end{proof}

We have now fully elaborated the decouplings to be utilized. All that remains is to introduce the rescalings. These in tandem with the linear maps of Lemma \ref{l2} complete a cycle of operations to be iterated. Naturally, it is simplest to narrate the whole procedure for low-dimensional $\M$, so we do this first for $\M^4$. Here, we do not have to be as ``efficient" as in Proposition \ref{pr4.1}. For example, we can take larger constant coefficients in \eqref{4.123}, that in particular exceed the component index $i$, and the final decoupling inequality is not worsened in any essential way. The only negative impact derives from the use of H\"older's inequality directly after each iterative step \eqref{4.124}, \eqref{4.125},... to reduce the slightly enlarged caps $\D^{(1)}, \D^{(2)}, \dots$

Let us reiterate. The reason for the meticulous care exercised in Proposition \ref{pr4.1} is to keep the coefficients $(i)_j$ in \eqref{1.01} from entering into the decoupling constant $\k$. But the size of these coefficients is mitigated when $i$ is kept low, and then their effect is harmless.

\section{Proof of Theorem \ref{t1}} \label{s4dim} 

Fix $k_1, k_2 \in \N$. In light of Lemma \ref{le1}, we may assume that \begin{equation} \label{5.1} 2^{-k_i} > \delta^{i/4} \end{equation} for some $i$. Our goal is to deduce a decoupling partition of $\Ac_{k_1, k_2}$ by maximally flat caps $\D_{k_1, k_2}$. For this, we shall need to execute a case analysis, according as $2k_1 \geq k_2$ or otherwise. \\

In the sequel, we shall obtain caps of varying sizes, alternately through decoupling and rescaling. To keep track of them, we introduce the notation $\D^{(j_1, j_2)},$ $j_1$ being the number of rescalings applied prior to acquiring $\D^{(j_1, j_2)}$ while $j_2$ denotes the number of decoupling processes executed since the most recent rescaling. \\

\begin{proof} $ $\newline

\noindent \underline{$2k_1 \geq k_2$}: The hypothesis confirms that $s_1^2$ is essentially smaller than $s_2$ throughout $\Ac_{k_1, k_2}$. Thus, we find that $$\max_{\x(t,s) \in \Ac_{k_1, k_2}} |(t^2 + 2ts_1 + 2s_2) - (t+s_1)^2| = \max_{\x(t,s) \in \Ac_{k_1, k_2}} |2s_2 - s_1^2| \sim 2^{-k_2}.$$ It follows that $\Ac_{k_1, k_2} \subset \Pf^1\big(O(2^{-k_2})\big)$. Inputting $E= O(2^{-k_2})$, Lemma \ref{lem4.0} provides an initial decoupling over caps $\D^{(0,1)}$ of $t$-length $2^{-k_2/2}$: \begin{equation} \label{5.2} \|P_{\Nc_\delta(\Ac_{k_1, k_2})} f\|_p \lesssim \Dec\big(2^{-k_2}\big)(\sum_{\D^{(0,1)} \subset \Ac_{k_1, k_2}} \|P_{\Nc_\delta(\D^{(0,1)})} f\|_p^2)^{1/2}. \end{equation} 

Fix $\D^{(0,1)}$, and translate it to starting position. The $t$ variable is now bounded by $2^{-k_2/2}$, so we may apply a corresponding rescaling. Our goal is to remain within $\M$, while simultaneously enlarging $s_2$ to scale $\sim 1$. In this way, we replace $\D^{(0,1)}$ with standardized caps to which Proposition \ref{pr4.1} may be applied.  

Let us further partition $\D^{(0,1)}$. We decompose $[2^{-k_2}, 2^{-k_2 + 1})$ into subintervals $J$ of length approximately $2^{-k_2}\delta^\e$, and then partition $\D^{(0,1)}$ into sets $$\D^{(0,2)} = \x([0, 2^{-k_2/2}] \times [2^{-k_1}, 2^{-k_1 + 1}) \times J).$$ By H\"older's inequality, \begin{equation} \label{5.3} \|P_{\Nc_\delta(\D^{(0,1)})} f\|_p \leq \delta^{-\e/2}(\sum_{\D^{(0,2)} \subset \D^{(0,1)}} \|P_{\Nc_\delta(\D^{(0,2)})} f\|_p^2)^{1/2}. \end{equation} 

To each $\Delta^{(0,2)}$, we apply the rescaling \begin{equation} \label{5.4} \Db(\xi_1, \dots, \xi_4) = \big((3a)^{-1/2}\xi_1, (3a)^{-1}\xi_2, (3a)^{-3/2}\xi_3, (3a)^{-2}\xi_4 \big)\end{equation} where $a$ is some number in the corresponding interval $J$. Relabeling \begin{equation} \label{5.41} t' = (3a)^{-1/2} t, \qquad s_1' = (3a)^{-1/2} s_1, \qquad s_2' = (3a)^{-1} s_2,\end{equation} we see that $\Db$ maps $\D^{(0,2)}$ to the set \begin{equation} \label{5.5} \D^{(1,0)} = \{(t' + s_1', t'^2 + 2t's_1' + 2s_2', t'^3 + 3t'^2s_1' + 6t's_2', t'^4 + 4t'^3s_1' + 12t'^2s_2'): \end{equation} $$  0 \leq t' \leq 1, s_1' \sim 2^{k_2/2-k_1}, s_2' \in (3a)^{-1}J\},$$ while enlarging the neighborhood width from $\delta_0 = \delta$ to $$\delta_1 = (3a)^{-2}\delta \sim 2^{2k_2}\delta.$$ The transition from $\D^{(0,2)}$ to $\D^{(1,0)} = \Db(\D^{(0,2)})$ is embodied in the equality \begin{equation} \label{5.6} \|P_{S} f\|_p = |\det \Db|^{-1/p'}\|P_{\Db(S)} g\|_p, \qquad (\hat{g} \coloneqq \hat{f} \circ \Db^{-1})  \end{equation} valid for all subsets $S \subset \Nc_\delta(\D^{(0,2)})$ and all functions $f$.

Note that within $\D^{(1,0)}$ \begin{equation} \label{5.51} |s_2' - 1/3| \leq \delta^\e \ \end{equation} since $a \in J$, and $$c_3 = \frac{1}{3}$$ according to \eqref{4.1190}. Also, $s_1'$ is bounded by case hypothesis. Thus, the caps $\D^{(1,1)}$ of $t'$-length $\delta_1^{\: \e}$ partitioning $\D^{(1,0)}$ are standardized with respect to the third component: \begin{IEEEeqnarray*}{rCl} \|P_{\Nc_{\delta}(\D^{(0,2)})} f\|_p &=& |\det \Db|^{-1/p'}\|P_{\Nc_{\delta_1}(\D^{(1,0)})} g\|_p \\  &\leq& \delta_1^{-\e/2} |\det \Db|^{-1/p'}  (\sum_{\D^{(1,1)} \subset \D^{(1,0)}} \|P_{\Nc_{\delta_1}(\D^{(1,1)})} g\|_p^2 )^{1/2}. \end{IEEEeqnarray*}

Now we are ready to complete the argument for this case. So far, we have determined \begin{equation} \label{5.61} \|P_{\Nc_\delta(\Ac_{k_1, k_2})} f\|_p \leq \delta^{-\e} \Dec\big(2^{-k_2}\big) \Big(\sum_{\D^{(1, 1)}} (|\det \Db|^{-1/p'} \|P_{\Nc_{\delta_1} (\D^{(1,1)})} g\|_p)^2 \Big)^{1/2}. \end{equation} Observe the relation between the third and fourth components of \eqref{5.5}: over each $\D^{(1,1)}$ translated to starting position, \begin{IEEEeqnarray*}{rCl} \xi_4 \circ \x - (\xi_3 \circ \x)^2 &=& (t'^4 + 4t'^3s_1 + 12t'^2s_2') - (t'^3 + 3t'^2s_1' + 6t's_2')^2 \\ &=& 12s_2'(1 - 3s_2')t'^2 + O(t'^3) \\   &\lesssim& \delta_1^{\: \e} t'^2. \IEEEyesnumber \label{5.7} \end{IEEEeqnarray*} The inequality follows from the fact that the $\D^{(1,1)}$ are standardized. Thus, the $\delta_1$-neighborhood of any cap $\x([0,T] \times I_1 \times I_2) \subset \D^{(1,1)}$ is contained in $\Pf^3(E)$ with $$E = O(\delta_1^{\: \e} T^2) + \delta_1.$$

 The calculation \eqref{5.7} was already done in Subsection \ref{ss4.1}, but we repeated it here to show how it naturally occurs in the course of successive decouplings. As well, unencumbered by arbitrarily large dimensional constants, the calculation is much simpler in 4 dimensions. Proposition \ref{pr4.1} now applies. Taking $i=3, n =4$, the proposition confirms that the caps $\D^{(1,2)} \subset \D^{(1,1)}$ of $t$-length $\delta_1^{1/2}$ occasion the following inequality:  \begin{equation} \label{5.10} \|P_{\Nc_{\delta_1}(\D^{(1,1)})} g\|_p \leq (30 \delta_1^{-\e})^{1/2} \big(\k \Dec \big(\delta_1 \big)\big)^{1/\e} \big(\sum_{\D^{(1,2)} \subset \D^{(1,1)}} \|P_{\Nc_{\delta_1}(\D^{(1,2)})} g\|_p^2 \big)^{1/2}. \end{equation}
Recall that \eqref{5.6} implies that \begin{equation} \label{5.13} \|P_{\Nc_{\delta_1}(\D^{(1,2)})} g\|_p = (\det \Db)^{1/p'} \|P_{\Db^{-1}(\Nc_{\delta_1}(\D^{(1,2)}))} f\|_p. \end{equation} In conclusion, we obtain the desired decoupling inequality \eqref{1.02} for the inputted function $P_{\Nc_\delta(\Ac_{k_1, k_2})} f$ through the combination of \eqref{5.61}, \eqref{5.10}, and \eqref{5.13}. (Monotonicity of the decoupling constant is to be used, since $2^{-k_2} \geq \delta^{1/2}$ and $\delta_1 \geq \delta$.)  We claim that the preimages $\Db^{-1}(\D^{(1,2)})$ are indeed the caps $\D_{k_1, k_2}$ of Theorem \ref{t1}. Observing the change of variables \eqref{5.41} for $t'$, we deduce that their $t$-length is indeed $$2^{-k_2/2}\delta_1^{1/2} = 2^{-k_2/2}(2^{2k_2}\delta)^{1/2} = 2^{k_2/2}\delta^{1/2}.$$ 

Each $\D_{k_1, k_2}$ is almost flat. To verify this claim, note first that the $t^4$ term in $$\xi_4 \circ \x = t^4 + 4t^3s_1 + 12t^2s_2$$ is surely $O(\delta)$ since $(2^{k_2}\delta)^{1/2}$ can at most be $\delta^{1/4}$ (recall \eqref{5.1}). Concerning the $s_1 t^3$ term, observe that our case hypothesis implies $$ (1/2)\max_{\x(t,s) \in \D_{k_1, k_2}} s_1 t^3 = 2^{-k_1}(2^{k_2}\delta)^{3/2} \leq 2^{-k_2/2}(2^{k_2}\delta)^{3/2} = 2^{k_2}\delta^{3/2} \leq \delta,$$ the final inequality being true by \eqref{5.1}. And of course, the following is immediate $$ (1/2)\max_{\x(t,s) \in \D_{k_1, k_2}} s_2 t^2 = 2^{-k_2} (2^{k_2}\delta) = \delta.$$

\noindent \underline{$k_2 > 2k_1$}: This case also makes use of the same parabolic cylindrical decoupling as in the previous case. However, handling the new hypothesis requires an additional preliminary step. Now it holds that $s_1^2$ is essentially larger than $s_2$, so that instead $$\max_{\x(t,s) \in \Ac_{k_1, k_2}} |(\xi_2 \circ \x) - (\xi_1 \circ \x)^2| \sim 2^{-2k_1}.$$ Consequently, Lemma \ref{lem4.0} provides the inequality \begin{equation} \label{5.15} \|P_{\Nc_\delta(\Ac_{k_1, k_2})} f\|_p \lesssim  \Dec \big(2^{-2k_1} \big) \big(\sum_{\D^{(0,1)} \subset \Ac_{k_1, k_2}} \|P_{\Nc_\delta(\D^{(0,1)})} f\|_p^2 \big)^{1/2}, \end{equation} true for caps $\D^{(0,1)}$ of $t$-length $2^{-k_1}$ partitioning $\Ac_{k_1, k_2}$. Observe that $2^{-2k_1} \geq \delta^{1/2}$ by hypothesis, so \eqref{5.15} is indeed a decoupling.

Translate $\D^{(0,1)}$ to starting position, so that we may appropriately rescale. For this, it is prerequisite as before to partition $\D^{(0,1)}$ into caps $$\D^{(0,2)} = \x([0, 2^{-k_1}] \times J \times [2^{-k_2}, 2^{-k_2 + 1}))$$ where $\{J\}$ is a partition of $[2^{-k_1}, 2^{-k_1 +1})$ into intervals of length approximately $2^{-k_1}\delta^\e$. By H\"older's inequality, \begin{equation} \label{5.16} \|P_{\Nc_\delta(\D^{(0,1)})} f\|_p \leq \delta^{-\e/2} (\sum_{\D^{(0,2)} \subset \D^{(0,1)}} \|P_{\Nc_\delta(\D^{(0,2)})} f\|_p^2)^{1/2}. \end{equation} For each $\D^{(0,2)}$, select $a \in J$, and apply the dilation \begin{equation} \label{5.17} \Db_1(\xi_1, \xi_2, \xi_3, \xi_4) = \big(3(4a)^{-1}\xi_1, 9(4a)^{-2}\xi_2, 27(4a)^{-3}\xi_3, 81(4a)^{-4} \xi_4\big) \end{equation} to $\D^{(0,2)}$. Relabeling \begin{equation} \label{5.18} t' = 3(4a)^{-1}t, \qquad s_1' = 3(4a)^{-1}s_1, \qquad s_2' = 9(4a)^{-2} s_2, \end{equation} $$J^{(1)} = 3(4a)^{-1}J, \qquad k_2' = k_2 - 2k_1 > 0,$$ the image is the rescaled cap $$\D^{(1,0)} = \{(t' + s_1', t'^2 + 2t's_1' + 2s_2', t'^3 + 3t'^2s_1' + 6t's_2', t'^4 + 4t'^3s_1' + 12t'^2s_2'):$$ \begin{equation} \label{5.19}  0 \leq t' \leq 1, s_1' \in J^{(1)}, s_2' \in \Ic_{k_2'}\} \end{equation} with neighborhood width enlarged from $\delta$ to $$\delta_1 = 81(4a)^{-4}\delta \sim 2^{4k_1}\delta.$$ Observe that subject to \eqref{5.19} $$|s_1' - 3/4| \lesssim \delta^\e$$ which is meaningful since $$c_2 = 3/4$$ by \eqref{4.1190}. As well, \begin{equation} \label{5.20} \|P_{S} f\|_p = (\det \Db_1)^{-1/p'} \|P_{\Db_1(S)} g_1\|_p, \qquad (\hat{g_1} \coloneqq \hat{f} \circ \Db_1^{-1}) \end{equation} for all subsets $S \subset \Nc_\delta(\D^{(0,2)})$. In particular, $$ \|P_{\Nc_\delta(\D^{(0,2)})} f\|_p = (\det \Db_1)^{-1/p'}\|P_{\Nc_{\delta_1}(\D^{(1,0)})} g_1\|_p. $$

Shifting our attention now to $\|P_{\Nc_{\delta_1}(\D^{(1,0)})} g_1\|_p$, we are ready to reduce the cap sizes to the next lower scale, one that is indexed by $s_2'$. The goal here is to find caps partitioning $\D^{(1,0)}$ of sufficiently small size that can be rescaled to caps with $s_2'$ parameter being essentially 1. Then, we may conclude the argument as done previously.

Let us partition $\D^{(1,0)}$ by caps $\D^{(1,1)}$ of $t$-length $\delta_1^\e$: \begin{equation} \label{5.202} \|P_{\Nc_{\delta_1}(\D^{(1,0)})} g_1\|_p \leq  \delta_1^{-\e/2}(\sum_{\D^{(1,1)} \subset \D^{(1,0)}} \|P_{\Nc_{\delta_1}(\D^{(1,1)})} g_1\|_p^2)^{1/2}. \end{equation} Each $\D^{(1,1)}$ is standardized relative to the second component. Consider the consequent relation between the second and third components of $\x$ throughout $\D^{(1,1)}$ \begin{IEEEeqnarray*}{rCl} |\xi_3 \circ \x - (\xi_2 \circ \x)^2| &=& |(t'^3 + 3t'^2s_1' + 6t's_2') - (t'^2 + 2t's_1' + 2s_2')^2| \\ &=& |s_1'(3 - 4s_1')t'^2 - 4s_2'^2| + O(t'^3) + |s_2'|O(t') \\ &\leq& O(\delta_1^{\: \e})t'^2 + O(2^{-2k_2'}) + O(2^{-k_2'})t'. \IEEEyesnumber \end{IEEEeqnarray*} This is indeed the bound \eqref{4.123} promised by Proposition \ref{pr4.1} with $L = O(2^{-k_2'})$ \big(see the remark following \eqref{4.1220}\big), and the proposition confirms that there is a decoupling over caps $\D^{(1,2)} \subset \D^{(1,1)}$ of $t$-length $2^{-k_2'}$:  \begin{equation} \label{5.25} \|P_{\Nc_{\delta_1}(\D^{(1,1)})} g_1\|_p \lesssim \delta_1^{-\e/2} \big(\k \Dec \big(2^{-2k_2'} \big)\big)^{2\e^{-1}\log_{\delta_1}{2^{-k_2'}}} \Big(\sum_{\D^{(1,2)} \subset \D^{(1,1)}} \|P_{\Nc_{\delta_1}(\D^{(1,2)})} g_1\|_p^2 \Big)^{1/2}. \end{equation} Note that $2^{-k_2'} \geq \min\{2^{-k_2}, \delta_1\} \geq \min\{\delta^{1/2}, \delta_1\}.$ Thus, $\Dec(2^{-2k_2'}) \leq \Dec(\delta)$ in particular.

Since $\delta_1 \geq \delta$, the juxtaposition of \eqref{5.20}, \eqref{5.202}, and \eqref{5.25} determines \begin{IEEEeqnarray*}{rCl} \|P_{\Nc_\delta(\D^{(0,2)})} f\|_p &\lesssim& \delta^{-\e} \big(\k \Dec \big(\delta \big) \big)^{2/\e} (\det \Db_1)^{-1/p'} \Big(\sum_{\D^{(1,2)} \subset \Db_1(\D^{0, 2})} \|P_{\Nc_{\delta_1}(\D^{(1,2)})} g_1\|_p^2 \Big)^{1/2} \\ &=&  \delta^{-\e} \big(\k \Dec \big(\delta \big) \big)^{2/\e}\Big(\sum_{\Db^{-1}(\D^{(1,2)}) \subset \D^{0, 2}} \|P_{\Nc_{\delta}(\Db^{-1}(\D^{(1,2)}))} f\|_p^2 \Big)^{1/2} \IEEEyesnumber \label{5.30} \end{IEEEeqnarray*} By \eqref{5.18}, the length of each $\D = \Db_1^{-1}(\D^{(1,2)})$ is $$2^{-k_1}2^{-k_2'} = 2^{-k_1}2^{2k_1 - 2k_2} = 2^{k_1 - k_2},$$ and this actually can be the scale at which caps in $\Ac_{k_1, k_2}$ are maximally almost flat. Inspired from the minimum value mentioned in Theorem \ref{t1}, the inequality $$2^{k_1 - k_2} < (2^{k_2}\delta)^{1/2}, $$ equivalently \begin{equation} \label{5.32} 2^{2k_1 -3k_2} < \delta, \end{equation} is readily checked to be equivalent to $$2^{-k_1} \max_{\x(t,s) \in \D} t^3 = 2^{2k_1 - 3k_2} = 2^{-k_2} \max_{\x(t,s) \in \D} t^2 < \delta,$$ proving almost flatness of $\D$. In particular, \eqref{5.32} holds when \begin{equation} \label{5.300} 2^{-k_2} \leq \delta^{1/2}.\end{equation} Applying \eqref{5.30} to each term on the right side of \eqref{5.16}, together with \eqref{5.15}, we have in this case \begin{equation} \label{5.31} \|P_{\Nc_\delta(\Ac_{k_1, k_2})} f\|_p \leq \delta^{-3\e/2} \Dec\big(\delta^{1/2} \big) \big(\k\Dec\big(\delta \big)\big)^{2/\e} \Big(\sum_{\D \subset \Ac_{k_1, k_2}} \|P_{\Nc_{\delta}(\D)} f\|_p^2 \Big)^{1/2}, \end{equation} as desired.

Therefore, we assume that \eqref{5.32} is false and show that in this scenario our argument extends to decoupling cap sizes of scale $(2^{k_2}\delta)^{1/2}$. Indeed, this scale marks flatness since the following inequality holds by assumption $$2^{-k_1} (2^{k_2} \delta)^{3/2} \leq \delta.$$

Since \eqref{5.32}, and therefore \eqref{5.300}, is false, we recall from \eqref{5.19} that $s_2' \sim 2^{-k_2'},$ which is the $t$-length of $\D^{(1,2)}$. Thus, $s_2'$ may be rescaled to size $\sim 1$. Translate $\D^{(1,2)}$ to starting position. Decomposing $[2^{-k_2'}, 2^{-k_2' + 1})$ into subintervals $J$ of length approximately $2^{-k_2'}\delta_1 ^{\: \e}$, let $$\D^{(1,3)} = \x([0, 2^{-k_2'}] \times J^{(1)} \times J)$$ be the corresponding caps that partition $\D^{(1,2)}$. The relevant decoupling inequality is \begin{equation} \label{5.33} \|P_{\Nc_{\delta_1}(\D^{(1,2)})} g_1\|_p \leq \delta_1^{-\e/2} (\sum_{\D^{(1,3)} \subset \D^{(1,2)}} \|P_{\Nc_{\delta_1}(\D^{(1,3)})} g_1\|_p^2)^{1/2}. \end{equation} 

We are imitating the procedure carried out in the first case ($2k_1 \geq k_2$), yet the dilation we use here must be different. For each $\D^{(1, 3)}$, select $a \in J$, and apply the map $$\Db_2(\xi_1, \xi_2, \xi_3, \xi_4) = (\xi_1, (3a)^{-1}\xi_2, (3a)^{-2}\xi_3, (3a)^{-3}\xi_4)$$ to $\D^{(1,3)}$. Relabeling \begin{equation} \label{5.34} t'' = (3a)^{-1}t', \qquad s_1'' = s_1', \qquad s_2'' = (3a)^{-1}s_2' \end{equation} $$J^{(2)} = (3a)^{-1}J, $$ the image is $$\D^{(2, 0)} = \{(3at'' + s_1'', 3a(t'')^2 + 2t''s_1'' + 2s_2'', 3a(t'')^3 + 3(t'')^2s_1'' + 6t''s_2'', $$ \begin{equation} \label{5.35} 3a(t'')^4 + 4(t'')^3s_1'' + 12(t'')^2 s_2''):  0 \leq t'' \leq 1, s_1'' \in J^{(1)}, s_2'' \in J^{(2)}\}\end{equation} with $s_2''$ rescaled to $\sim 1$ and neighborhood width enlarged from $\delta_1$ to \begin{equation} \label{5.350} \delta_2 = (3a)^{-3}\delta_1 \sim 2^{3k_2'}\delta_1 \sim 2^{3(k_2 - 2k_1)}2^{4k_1}\delta = 2^{3k_2 - 2k_1}\delta.\end{equation} Indeed, $$\delta_2 \leq 1$$ since \eqref{5.32} is false. For all subsets $S \subset \Nc_{\delta_1}(\D^{(1,3)})$,  \begin{equation} \label{5.351} \|P_S g_1\|_p = (\det \Db_2)^{-1/p'}\|P_{\Db_2(S)} g_2\|_p, \qquad (\hat{g_2} \coloneqq \hat{g_1} \circ \Db_2^{-1}). \end{equation}

Observe that $\D^{(2,0)}$ is not parametrized by $\x = \x_{ext}(\cdot, 1, \cdot)$, but it is parametrized by $\x_{ext}(\cdot, 3a, \cdot)$. Therefore, Lemma \ref{l2} applies to $\D^{(2,0)}$, making it possible to apply our iterative decoupling machinery. To get started, we observe that $$\xi_4 \circ \x_{ext} - (\xi_3 \circ \x_{ext})^2 = O((t'')^3) + 12s_2''(3s_2'' - 1)(t'')^2$$ \begin{equation} \label{5.36} = O((t'')^3) + O(\delta_1^{\:\e})(t'')^2 \end{equation} by \eqref{5.35}, which gives the analogue of \eqref{5.51}. Once we partition $\D^{(2,0)}$ by caps $\D^{(2,1)}$ of $t''$-length $\delta_2^\e$ \begin{equation} \label{5.37} \|P_{\Nc_{\delta_2}(\D^{(2,0)})} g_2\|_p \leq \delta_2^{-\e/2}(\sum_{\D^{(2,1)} \subset \D^{(2,0)}} \|P_{\Nc_{\delta_2}(\D^{(2,1)})} g_2\|_p^2)^{1/2}, \end{equation} and translate each $\D^{(2,1)}$ to starting position, equation \eqref{5.36} will then be bounded by $O(\delta_2^{\: \e})(t'')^2$ over $\D^{(2,1)}$.  It follows that all subsets $$\Nc_{\delta_2}\big(\x_{ext}([0,T] \times \{3a\} \times J^{(1)} \times J^{(2)}) \big) \subset \Nc_{\delta_2}(\D^{(2,1)})$$ are contained respectively in parabolic cylindrical neighborhoods $\Pf^{(3)}$ of width $$O(\delta_2^{\:\e})(T)^2 + \delta_2.$$ The rest of Proposition \ref{pr4.1} applies, providing caps $\D^{(2,2)}$ of $t''$-length $\delta_2^{1/2}$ partitioning $\D^{(2,1)}$ such that \begin{equation} \label{5.39} \|P_{\Nc_{\delta_2}(\D^{(2,1)})} g_2\|_p \lesssim  \delta_2^{-\e} \big(\k \Dec \big(\delta_2 \big)\big)^{1/\e} \Big(\sum_{\D^{(2,2)} \subset \D^{(2,1)}} \|P_{\Nc_{\delta_2}(\D^{(2,2)})} g_2\|_p^2 \Big)^{1/2}. \end{equation} The desired decoupling inequality \eqref{1.02} is the result of combining \eqref{5.15}, \eqref{5.16}, the inequality in \eqref{5.30}, \eqref{5.33}, \eqref{5.351}, \eqref{5.37}, and \eqref{5.39} followed by a change of variables back to original coordinates using \eqref{5.20} and \eqref{5.351}. By \eqref{5.18}, \eqref{5.34}, and \eqref{5.350}, the caps $\D_{k_1, k_2} = \Db_1^{-1}(\Db_2^{-1}(\D^{(2,2)}))$ thus obtained have $t$-length $$ 2^{-k_1}2^{-k_2'}\delta_2^{1/2} \sim 2^{-k_1}2^{2k_1 - k_2}(2^{3k_2 - 2k_1}\delta)^{1/2} = (2^{k_2}\delta)^{1/2}.$$

\end{proof}

\section{Proof of Theorem \ref{t0.1}} \label{sgen}

In this section, $\M = \M^{n}$, the $n$-dimensional moment surface in $\R^{n+1}$. The $\ell^2$ decoupling for $\M$ shall be obtained by way of an inductive procedure that generalizes Section \ref{s4dim}. In this section, we outline that procedure.

In Section \ref{s4dim}, some rescalings of the moment surface took place there, resulting in a renaming of the $s_i$ variables. Conveniently, at most two rescalings occurred at a time, so trivial changes in the notation sufficed. However, here it will be necessary to introduce the notation $s_i^{(j)}$ to denote the new $s_i$ variables that describe a cap $\D^{(j, l)}$ obtained after $j$ rescalings. 

Let us begin. In light of \eqref{beg} and Lemma \ref{le1}, it suffices to deduce $\Pc_\delta(\Ac^\a_{k_1, \dots, k_{n-1}})$ where we assume that \begin{equation} \label{7.0} 2^{1-k_i} > (1/i!)\delta^{i/(n+1)} \end{equation} for at least one $i$. Initially, we just have the cap $\D^{(0,0)}$ that is the whole surface $\Ac^\a_{k_1, \dots, k_{n-1}},$ which we abbreviate to $\Ac$. Consider the ordering of the set $$\Sc^0 = \{2^{-k_1}, (2!2^{-k_2})^{1/2}, (3!2^{-k_3})^{1/3}, \dots, ((n-1)!2^{-k_{n-1}})^{1/(n-1)}\}.$$ Let $i_0$ be the index of the maximum element in $\Sc^0$. Then, \begin{equation} \label{7.1} (i_0!2^{-k_{i_0}})^{-i/i_0} 2^{-k_i} \lesssim 1/i! \end{equation} holds for all $i$. We shall need \eqref{7.1} in order to rescale effectively.

The first decoupling that we apply is with respect to the parabolic cylinder $\Pf^{(1)}$. We choose this one since the first component of $\x$ is just a linear function of $t,s$. As before, $$\xi_2 \circ \x -  (\xi_1 \circ \x)^2 = 2s_2 - s_1^2,$$ and both of the latter terms are essentially less than $(i_0!s_{i_0})^{2/{i_0}}$ by our choice of $i_0$. Therefore, $$\xi_2 \circ \x - (\xi_1 \circ \x)^2 = O((i_0!s_{i_0})^{2/{i_0}}),$$ so we may apply Lemma \ref{lem4.0} with neighborhood width $O((i_0!2^{-k_{i_0}})^{2/i_0})$ to obtain caps $\D^{(0,1)}$ partitioning $\D^{(0,0)}$ such that \begin{equation} \label{7.2} \|P_{\Nc_\delta(\D^{(0,0)})} f\|_p \lesssim \Dec \big((i_0!2^{-k_{i_0}})^{2/i_0} \big) \big(\sum_{\D^{(0,1)} \subset \D^{(0,0)}} \|P_{\Nc_\delta(\D^{(0, 1)})} f\|_p^2 \big)^{1/2}. \end{equation} The caps $\D^{(0,1)}$ have $t$-length $(i_0!2^{-k_{i_0}})^{1/i_{0}}$ and $s_i$-length $2^{-k_i}$.

In preparation for the next decoupling, let us divide $[2^{-k_{i_0}}, 2^{-k_{i_0} + 1})$ into subintervals $J$ of length $2^{-k_{i_0}}\delta^\e/30$. This is reflected in the inequality: \begin{equation} \label{7.21} \|P_{\Nc_\delta(\D^{(0, 1)})} f\|_p \leq (30\delta^{-\e})^{1/2} (\sum_{\D^{(0,2)} \subset \D^{(0,1)}} \|P_{\Nc_\delta(\D^{(0, 2)})} f\|_p^2)^{1/2} \end{equation} where each $\D^{(0, 2)}$ has $t$-length $(i_0!2^{-k_{i_0}})^{1/i_{0}}$ and $s_i$-length $2^{-k_i}$ for $i \ne i_0$ while $|s_{i_0}|$ ranges over the interval $J$.

Let $a \in J$, and label $$\Big(\frac{(i_0+2)_{i_0}}{a(i_0 + 1)_{i_0}^2}\Big)^{1/i_0}= \Big(\frac{i_0 +2}{2a(i_0+1)!} \Big)^{1/i_0} $$ as $d_1$. Note that $$\Big (\frac{2^{-1+k_{i_0}}}{i_0!} \Big)^{1/i_0} \leq d_1 \leq \Big (\frac{2^{k_{i_0}}}{i_0!} \Big)^{1/i_0}$$ and $$d_1= \Big(\frac{c_{i_0+1}}{a} \Big)^{1/i_0}$$ according to \eqref{4.1190}. To each $\D^{(0, 2)}$, we apply the rescaling \begin{equation} \label{dil} \Db_1(\xi_1, \dots, \xi_{n+1}) = (d_1\xi_1, d_1^{\,2} \xi_2, \dots, d_1^{ \,n+1}\xi_{n+1}).\end{equation} There is a corresponding change of variables: \begin{equation} \label{7.210} \|P_S f\|_p = (\det \Db_1)^{-1/p'} \|P_{\Db(S)} g_1\|_p, \qquad (\hat{g_1} \coloneqq \hat{f} \circ \Db_1^{-1}), \end{equation} true for all $S \subset \Nc_\delta(\D^{(0,2)}).$ Let $$s_i^{(1)} = d_1^{\;i} s_i \sim (2^{k_{i_0}}/i_0!)^{i/i_0}s_i.$$ Observe that $|s_i^{(1)}| \lesssim 1/i!$ in light of \eqref{7.1}. After translation to starting position, $\D^{(1, 0)} = \Db_1(\D^{(0, 2)})$ has the form \begin{equation} \label{7.22} \{\x(t, s_1^{(1)}, \dots, s_{n-1}^{(1)}): 0 \leq t \leq 1, s_i^{(1)} \sim 2^{(k_{i_0}/i_0)i - k_i}, |s_{i_0}^{(1)} - c_{i_0+1}| \leq \delta^\e/(30i_0!)\}\end{equation} and enlarged neighborhood width $$\delta_1 = d_1^{\,n+1}\delta \sim (2^{k_{i_0}}/i_0!)^{(n+1)/i_0}\delta,$$ which is less than 1 by \eqref{7.0}. In particular, $|s_{i_0}^{(1)}| \lesssim 1/i_0!$ so the first two inequalities of condition $(\star)$ as specified in \eqref{4.1} are met. The caps $\D^{(1,1)}$ of $t$-length $\delta^{\: \e}_1/(30(i_0 +2))$ partitioning $\D^{(1,0)}$ are thus standardized: \begin{equation} \label{7.23} \|P_{\Nc_{\delta_1}(\D^{(1, 0)})} g_1\|_p \leq (30\delta_1^{-\e})^{1/2}\sqrt{i_0+2}  \big(\sum_{\D^{(1,1)} \subset \D^{(1,0)}} \|P_{\Nc_{\delta_1}(\D^{(1,1)})} g_1\|_p^2 \big)^{1/2}. \end{equation}

Let $a_i \in \R$ be approximate values of $|s_i^{(1)}|$ for all $i$, i.e. $a_i = (1/i!)\delta_1^{(i - i_0)/(n+1-i_0)}$ if $|s_i^{(1)}| \leq \delta_1^{(i-i_0)/(n+1-i_0)}/i! $, and otherwise, $a_i$ is chosen to be any number satisfying $$a_i \leq |s_i^{(1)}| \leq 2a_i.$$ In particular, note that if $s_i \leq (1/i!)\delta^{i/(n+1)}$, then $$s_i^{(1)} = d_1^{\,i}s_i \leq \frac{1}{i!}(d_1^{\,n+1}\delta)^{i/(n+1)} = \frac{1}{i!} \delta_1^{i/(n+1)} \leq \frac{1}{i!} \delta_1^{(i-i_0)/(n+1-i_0)}.$$ Formerly, we have taken $(1/i!)\delta^{i/(n+1)}$ to be the ``cutoff\," for a variable $s_i$ \big(see \eqref{1.20}\big), but here the exponent on $\delta$ must change in light of the following. Consider $$\Sc^1 = \{(i_0+1)!a_{i_0+1}, ((i_0+2)!a_{i_0+2})^{1/2}, \dots, ((n-1)!a_{n-1})^{1/(n-1 - i_0)}\}.$$ We choose this set instead of $\Sc^0$ because in particular we must avoid rescaling $s_{i_0}^{(1)}$, lest this term become excessively large. As before, we pick the index corresponding to the maximum element of $\Sc^1$, denoting that index here as $i_1$. Then, \begin{equation} \label{7.3} (i_1!a_{i_1})^{(i_0 - i)/(i_1 - i_0)}|s_i^{(1)}| \leq 2/i! \qquad \text{for all } i > i_0, \end{equation} so $(i_1!a_{i_1})^{(i_1 - i_0)^{-1}}$ should be our target $t$-length. 

Following the computation of Subsection \ref{ss4.1}, using \eqref{7.22}, we determine that $\D^{(1,1)}$ is contained in a parabolic cylindrical neighborhood $\Pf^{(i_0+1)}$ of thickness $$\xi_{i_0+2} \circ \x - (\xi_{i_0 + 1 } \circ \x)^2 = O((i_0+1)^2)\delta^\e t^2 + O((i_0+1)(i_0 +1)!)s_{i_0 + 1}^{(1)} t + O((i_0+2)!)s_{i_0+2}^{(1)} $$ $$+ \,O((i_0+1)!)^2(s_{i_0 + 1}^{(1)})^2).$$ It is true that $$(i_1! a_{i_1})^{(i_1 - i_0)^{-1}} \geq \max\{(i_0+1)!s_{i_0+1}^{(1)}, ((i_0 +2)!s_{i_0+2}^{(1)})^{1/2} \}.$$ Consequently, if the $t$-length $T$ of $\D^{(1,1)}$ exceeds $(\delta^{-\e}/30)(i_1! a_{i_1})^{(i_1-i_0)^{-1}}$, then $$T \geq (\delta^{-\e}/30) \max\{(i_0 +1)!s_{i_0+1}^{(1)}, ((i_0+2)!s_{i_0+2}^{(1)})^{1/2}\}.$$ So we enlarge $\Pf^{(i_0+1)}$ to thickness \begin{equation} \label{7.4}  O((i_0+1)^2)\delta^\e t^2 + O((i_0+1)^2) i_1!a_{i_1}^{(i_1 - i_0)^{-1}}\end{equation} and apply Proposition \ref{pr4.1} with $(i_0+1)L = (i_1!a_{i_1})^{(i_1 - i_0)^{-1}}$:
  \begin{equation} \label{7.5} \|P_{\Nc_{\delta_1}(\D^{(1,1)})} g_1\|_p \leq (30\delta_1^{-\e})^{1/2} \big( \k \Dec \big( 6(i_0+1)^2L^2 \big) \big)^{2\e^{-1}\log_{\delta_1}{L}} \Big(\sum_{\D^{(1,2)} \subset \D^{(1,1)}} \|P_{\Nc_{\delta_1}(\D^{(1,2)})} g_1\|_p^2 \Big)^{1/2}. \end{equation} 
  
 Recall that \begin{equation} \label{7.50} a_{i_1} \geq (1/i_1!)\delta_1^{(i_1-i_0)/(n+1-i_0)}\end{equation} and $\delta_1$ exceeds $\delta$, so \eqref{7.5} is a decoupling inequality. If equality holds in \eqref{7.50}, then $i!|s_{i}^{(1)}| \leq \delta_1^{(i-i_0)/(n+1-i_0)}$ for all $i > i_0$, implying that $$t^{n+1} + \sum_{j=1}^{n-1} (n+1)_j |s_j^{(1)}| t^{n+1-j}\bigg |_{t=(1/n)\delta_1^{1/(n+1-i_0)}}$$ \begin{IEEEeqnarray*}{rCl} &=& \Big(t^{n+1} + \sum_{j=1}^{i_0} (n+1)_j |s_j^{(1)}|t^{n+1-j} \Big) + \sum_{j=i_0+1}^{n-1} |s_j^{(1)}| t^{n+1-j}\bigg |_{t=(1/n)\delta_1^{1/(n+1-i_0)}} \\ &\leq& \delta_1 \sum_{j=0}^{i_0} \binom{n+1}{j}n^{-(n+1-j)} + \sum_{j=i_0+1}^{n-1} \binom{n+1}{j} \delta_1^{(j-i_0)/(n+1-i_0)}\delta_1^{(n+1-j)/(n+1-i_0)}n^{-(n+1-j)} \\ &\sim& \delta_1. \end{IEEEeqnarray*} The caps $\D^{(1,2)}$ in \eqref{7.5} would not have $t$-length $\delta^{(n+1-i_0)^{-1}}/n$, but rather $t$-length $\delta^{(n+1-i_0)^{-1}}/(i_0+1)$. However, we apply H\"older's inequality to produce such caps \begin{equation} \label{7.51} \|P_{\Nc_{\delta_1}(\D^{(1,2)})} g_1\|_p \leq \sqrt{n/(i_0+1)} \big(\sum_{\D^{(1,3)} \subset \D^{(1,2)}} \|P_{\Nc_{\delta_1}(\D^{(1,3)})} g_1\|_p^2 \big)^{1/2}\end{equation} and decoupling inequality \eqref{1.02} results from the concomitance of \eqref{7.51} with \eqref{beg}, \eqref{7.2}, \eqref{7.21}, \eqref{7.210}, \eqref{7.23}, and \eqref{7.5}. We would add each $\Db_1^{-1}(\D^{(1,3)})$ to $\Pc_\delta(\M)$ in this case.\\

By reducing $t$, we are steadily eliminating the possible neighborhood widths until only the appropriate rescaling of $\delta$ remains. The $t$-length of our caps is now such that $s_{i_1}^{(1)}$ may be rescaled to $\sim 1/i_1!$. This paves the way for imminent decouplings with respect to neighborhood widths involving appropriately rescaled $s_{i_1 + 1}, s_{i_1 + 2}, \dots, s_{n-1}, \delta$.

Let us assume strict inequality in \eqref{7.50} and continue the argument. As before, it behooves us to proceed with an elementary decoupling: \begin{equation} \label{7.6} \|P_{\Nc_{\delta_1}(\D^{(1,2)})} g_1\|_p \leq \delta_1^{-\e} (\sum_{\D^{(1,3)} \subset \D^{(1,2)}} \|P_{\Nc_{\delta_1}(\D^{(1,3)})} g_1\|_p^2)^{1/2} \end{equation} where each $\D^{(1,3)}$ corresponds to a distinct interval $J$, disjoint from the others, of length $\sim a_{i_1}\delta_1^\e/30$ contained in $[a_{i_1}, 2a_{i_1}]$. $\D^{(1,3)}$ has $t$-length $(i_0+1)^{-1}(i_1!a_{i_1})^{(i_1-i_0)^{-1}}$ inherited from $\D^{(1,2)}$. Once again, we pick some $a \in J$ and set $d_2$ to be $$d_2 = (c_{i_1 + 1}/a)^{(i_1-i_0)^{-1}} \sim (i_1!a_{i_1})^{-(i_1 -i_0)^{-1}} \gg 1.$$ The dilation that we define next generalizes \eqref{dil}. Let $\Db_2 : \R^{n+1} \rightarrow \R^{n+1}$ be the map defined component-wise \begin{equation} \label{7.7} \xi_i \mapsto d_2^{\,i - i_0}\xi_i.\end{equation} Note that for $i < i_0$, the power in \eqref{7.7} is negative, so these components are reduced. In particular, $s_1^{(1)}, \dots, s_{i_0-1}^{(1)}$ are all diminished, which has no bearing upon the sequel of our argument. This scenario appears inevitable because we must maintain a form of the caps that is amenable to Lemma \ref{l2}.   
 
Let us describe the image of $\Db_2$. Define \begin{IEEEeqnarray*}{rCl} s_i^{(2)} &=& d_2^{\,i - i_0}s_i^{(1)} \\ \delta_2 \; \; &=& d_2^{\,n+1 - i_0} \delta_1 > \delta_1, \end{IEEEeqnarray*} and $b_i$ to be approximate values of $s_i^{(2)}$ \begin{IEEEeqnarray*}{rCll} b_i &\leq& s_i^{(2)} \leq 2b_i & \qquad \text{ if } |s_i^{(2)}| > (1/i!)\delta_2^{(i-i_1)/(n+1-i_1)}\\ b_i &=& (1/i!) \delta_2^{(i-i_1)/(n+1-i_1)} & \qquad \text{ otherwise}. \end{IEEEeqnarray*} As before, if $s_i^{(1)} \leq (1/i!)\delta_1^{(i-i_0)/(n+1-i_0)}$, then $$s_i^{(2)} = d_2^{\,i-i_0}s_i^{(1)} \leq \frac{1}{i!}(d_2^{\,n+1-i_0}\delta_1)^{(i-i_0)/(n+1-i_0)} = \frac{1}{i!} \delta_2^{(i-i_0)/(n+1-i_0)} \leq \frac{1}{i!} \delta_2^{(i-i_1)/(n+1-i_1)}$$ for the same reason as previously, namely that $i < n+1.$ 

Each $s_i^{(2)}$ is $O(1/i!)$ by the comment following \eqref{7.7} and by \eqref{7.3}. After translation to starting position, each set $\Db_2(\D^{(1,3)})$ is $$\D^{(2,0)} = \{\x_{ext}(t, s_0^{(2)}, s_1^{(2)}, \dots, s_{n-1}^{(2)}) : 0 \leq t \leq 1, s_0^{(2)} = d_2^{-i_0}, s_i^{(2)} \sim b_i, $$ $$|s_{i_1}^{(2)} - c_{i_1 + 1}| \leq \delta_1^\e/(30i_1!)\},$$ and the neighborhood $\Nc_{\delta_1}(\D^{(1,3)})$ is enlarged to $\Nc_{\delta_2}(\D^{(2,0)})$. The terms on the right of \eqref{7.6} are impacted according to the following substitution: \begin{equation} \label{7.8} \|P_S g_1\|_p = (\det \Db_2)^{-1/p'} \|P_{\Db(S)} g_2\|_p \qquad \forall S \subset \Nc_{\delta_1}(\D^{(1,2)}), \qquad (\hat{g_2} \coloneqq \hat{g_1} \circ \Db_2^{-1}) .\end{equation} 

The caps $\D^{(2,1)} \subset \D^{(2,0)}$ of $t$-length $\delta_2^\e/(30(i_1+2))$ are standardized $$\|P_{\Nc_{\delta_2}(\D^{(2,0)})} g_2\|_p \leq (30\delta_2^{-\e})^{1/2}\sqrt{(i_1+2)/(i_0+2)} \Big(\sum_{\D^{(2,1)} \subset \D^{(2,0)}} \|P_{\Nc_{\delta_2}(\D^{(2,1)})} g_2\|_p^2 \Big)^{1/2},$$ so we are prepared for the next rescaling. This one is extracted from the smaller set $$\Sc^2 = \{(i_1+1)!b_{i_1 + 1}, \big((i_1+2)!b_{i_1 + 2}\big)^{1/2}, \dots, \big((n-1)!b_{n-1}\big)^{1/(n-1 - i_1)}\}.$$ As before with $\Sc^1$, we pick the maximal element of $\Sc^2$, label its index as $i_2$, and apply Proposition \ref{pr4.1} with $i=i_1 + 1, L= (i_1 + 1)^{-1}(i_2!b_{i_2})^{(i_2 - i_1)^{-1}}$: $$\|P_{\Nc_{\delta_2}(\D^{(2,1)})} g_2\|_p \leq (30\delta_2^{-\e})^{1/2} \big( \k \Dec \big( 6(i_1+1)^2L^2 \big) \big)^{2\e^{-1}\log_{\delta_2}{L}} \Big(\sum_{\D^{(2,2)} \subset \D^{(2,1)}} \|P_{\Nc_{\delta_2}(\D^{(2,2)})} g_2\|_p^2 \Big)^{1/2}.$$ Slightly smaller caps $\D^{(2,3)} \subset \D^{(2,2)}$ may be almost flat for scale $\delta_2$, depending on whether the maximal element $s_{i_2}^{(2)}$ is bounded by $(1/i_2! )\delta_2^{(i-i_1)/(n+1-i_1)}$ or not. In the former case, we add $(\Db_2 \circ \Db_1)^{-1}(\D^{(2,3)})$ to $\Pc_\delta(\M)$. If the latter holds, we rescale.
                                                                                                                                                                                                                                                                                                                                                                                                                                                                                                                                                                                             
We may repeatedly construct sets $\Sc^l$, decouple, and rescale just as outlined above. Each $l$-th step produces a comprehensive decoupling inequality  $$\|P_{\Nc_{\delta_l}(\D^{(l,0)})} g_l\|_p \leq (30\delta_l)^{-3\e/2}\sqrt{(i_{l}+2)/(i_{l-1} +2)}  \big( \k \Dec \big( 6(i_{l-1}+1)^2L_l^2 \big) \big)^{2\e^{-1}\log_{\delta_l}{L_l}} $$ \begin{equation}\label{7.10} \cdot \: \big(\sum_{\D^{(l,3)} \subset \D^{(l,0)}} \|P_{\Nc_{\delta_l}(\D^{(l,3)})} g_l\|_p^2 \big)^{1/2} \end{equation} and is followed by a rescaling, where \begin{IEEEeqnarray*}{rCl}   L_l &\geq& (i_{l-1}+1)^{-1}\delta_l^{1/(n+1-i_{l-1})} \\ \delta_l &=& d_l^{n+1 - i_{l-2}}\delta_{l-1}\\s^{(l)}_i &=& d_l^{n+1-i_{l-2}} s^{(l-1)}_i\\ \Db_l &:& \xi_i \mapsto d_l^{i-i_{l-2}}\xi_i \\ \hat{g}_l &=& \hat{g}_{l-1} \circ \Db_l^{-1} \end{IEEEeqnarray*} and $i_{l}$ is the maximal index for the set $$\Sc^l = \{(i_{l-1}+1)!|s^{(l)}_{i_{l-1} +1}|, \big((i_{l-1}+2)!|s^{(l)}_{i_{l-1}+2}|\big)^{1/2}, \cdots, \big((n-1)!|s^{(l)}_{n-1}|\big)^{1/(n-1-i_{l-1}}\},$$ $d_l$ being essentially the reciprocal of that maximal element. Therefore, generally the constant factor in \eqref{7.10} has the simpler upper bound$$(30\delta_l^{-3\e})^{1/2} \sqrt{(i_{l}+2)/(i_{l-1} +2)} \big( \k \Dec \big( \delta_l^{1/(n+1-i_{l-1})} \big) \big)^{2\e^{-1}[(n+1-i_{l-1})^{-1} + \log_{\delta_l}(i_{l-1}+1)^{-1}]} $$ which, since $\delta < \delta_l < 1$, may be further bounded by an expression involving just the original scale \begin{equation} \label{7.11}  (30\delta^{-3\e})^{1/2} \sqrt{(i_{l}+2)/(i_{l-1} +2)} \big( \k \Dec \big( \delta \big) \big)^{2\e^{-1}[(n+1-i_{l-1})^{-1} + \log(i_{l-1}+1)]}.\end{equation} 

The iterative process continues until a set $\Sc^N$ occurs for which either the $(n-1)$ index is maximal or all terms in $\Sc^N$ are less than $\delta_N^{1/(n+1-i_{N-1})}$. Note $N \leq n-1$. Indeed, one of these must occur as each step rescales an $i_l!s^{(l)}_{i_l}$ parameter to essentially 1, making it possible to restrict attention next to the $s^{(l)}_{i}$ satisfying $i > i_l$. After the $N$-th step, we have caps $\D^{(N,3)}$ whose $t$-length is $T_0 = (n-2)!s_{n-1}^{(N)}$ by Proposition \ref{pr4.1}. $T_0$ is the scale at which $\x_{ext}$ is almost flat in the $(n-1)$-component. More precisely, once $\D^{(N,3)}$ has been translated to starting position, \begin{equation} \label{7.81} \xi_{n-1}\Big|_{\D^{(N,3)}} = O\big((n-1)!|s_{n-1}^{(N)}|\big)\end{equation} by Lemma \ref{lemf}. This property of course is inherited by each cap contained within $\D^{(N,3)}$.

Moreover, the caps $\D^{(N,3)}$ may be almost flat relative to scale $\delta_N$. If they are not, then we rescale as customary to procure caps $\D^{(N+1, 0)}$ for which $(n-1)!s^{(N)}_{n-1} \sim 1.$ For the latter, we may consider the final parabolic approximation over each standardized cap $\D^{(N+1,1)} \subset \D^{(N+1,0)}$: \begin{equation} \label{7.9} \xi_{n+1} \circ \x - (\xi_n \circ \x)^2 = O(\delta^\e)t^2 + O(\delta_{N+1}).\end{equation} Applying Proposition \ref{pr4.1} to each $\D^{(N+1, 1)}$ yields a decoupling partition $\{\D^{(N+1, 2)}\}$ corresponding to $t$-intervals of length $(1/n)\delta_{N+1}^{1/2}$. But then, once translated to starting position, by Lemma \ref{lemf}, $\D^{(N+1, 2)}$ satisfies \begin{equation} \label{7.91}  \xi_{n+1}|_{\D^{(N+1, 2)}} = O(\delta_{N+1}) \end{equation} and $\D^{(N+1,2)}$ lies within $O(\delta_{N+1})$ of its tangent plane as desired. Undoing all of the rescalings, since these maps (as well as the translation maps $\A$) are linear and preserve $\M$ and therefore its tangent spaces, the preimage $\D$ of $\D^{(N+1,2)}$ will obey the same property relative to $\delta$. As well, noting that each rescaling $\Db_l^{-1}$ preserves inequalities of the form $$\xi_{i} \circ \x_{ext}(t, d_l^{-i_{l-2}}, |s^{(l)}_1|, \dots, |s^{(l)}_{n-1}|) = O(|i! s^{(l)}_i|),$$ we see that \eqref{7.81} and \eqref{7.91} imply that \begin{IEEEeqnarray}{rCl} \xi_{n-1} \circ \x(t, |s_1|, \dots, |s_{n-1}|) \Big|_{\D} &=& O\big((n-1)!|s_{n-1}|\big) \\ \xi_{n+1} \circ \x(t, |s_1|, \dots, |s_{n-1}|) \Big|_{\D} &=& O(\delta) \label{7.92}. \end{IEEEeqnarray} Thus, $\Delta$ will also have $t$-length $T$ essentially bounded by \begin{equation} \label{7.14} T \leq \min \big\{\min_i |(n-1-i)!s_{n-1}/s_i|^{1/(n-1-i)}, (2^{k_{n-1}}\delta/(n+1)!)^{1/2}\big\}.\end{equation} \big(Note that by Stirling's formula, $(n-1-i)!^{(n-1-i)^{-1}} \sim (n-1-i)^{3/2}$.\big) $\D$ is added to $\Pc_\delta(\M)$. Note that in all cases, the elements of our partition $\Pc_\delta(\M)$ have $t$-length bounded by $(2^{k_{n-1}}\delta/(n+1)!)^{1/2}$ since \eqref{7.91}, and therefore \eqref{7.92}, holds universally.

By \eqref{7.10} and \eqref{7.11}, the full decoupling inequality that we have obtained is \begin{equation}\label{7.12} \|P_{\Nc_\delta(\Ac^\a_{k_1, \dots, k_{n-1}})} f\|_p \leq (30\delta^{-3\e/2})^{n/2}n^{1/2} \big(\kappa \Dec_p^\P \big(\delta \big) \big)^{2\e^{-1}n\log(n)} \Big(\sum_{\D \subset \Nc_\delta(\Ac^\a_{k_1, \dots, k_{n-1}})} \|P_{\Nc_\delta(\D)} f\|_p^2 \Big)^{1/2}. \end{equation} This bound should be taken if $N \sim n$. If rather $N = O(1)$, \eqref{7.12} may be replaced with \begin{equation} \label{7.13} \|P_{\Nc_\delta(\Ac^\a_{k_1, \dots, k_{n-1}})} f\|_p \leq (30\delta^{-3\e/2})^{O(1)}n^{1/2} \big(\kappa \Dec_p^\P \big(\delta \big) \big)^{\e^{-1}O(\log(n))} \Big(\sum_{\D \subset \Nc_\delta(\Ac^\a_{k_1, \dots, k_{n-1}})} \|P_{\Nc_\delta(\D)} f\|_p^2 \Big)^{1/2}. \end{equation} Recall that we have the bound $\Dec_p^\P\big(\delta\big) \lesssim (\log 1/\delta)^{O(1)}$. The proof is complete.

\begin{re} It is possible to work just with $\delta$ in the pigeonholing of the $t$ and $s_{i_l}^{(l)}$ variables and the application of Proposition \ref{pr4.1}. If we do so, $\log(n)$ in the exponent on $\kappa \Dec^\P_p(\delta)$ in \eqref{7.13} may be replaced with 1 when $\delta < 1/n$. \end{re}

\section{Proof of Theorem \ref{t0}} \label{lastsec}

The moment surface $\M^n$ is a prototype for the general ruled hypersurface generated by a space curve. In this section, we extend the decoupling proved for $\M^n$ to the ruled hypersurface $\Mf^n$ generated by a nondegenerate curve $\phi: [-1,1] \rightarrow \R^{n+1}$. 

Recall that it is assumed that $\| \phi^{(i)}(t)\| \leq C$ for all $1\leq i \leq n+2$ and for all $t$. We say that $\M$ is prototypical for $\Mf$ because by Taylor approximation (and the Cauchy-Schwarz inequality), $$\phi^{(j)}(t) = \sum_{i=j}^{n+1} \frac{(t-t_0)^{i-j}}{(i-j)!} \phi^{(i)}(t_0) + O\Big(C\frac{|t-t_0|^{n+2-j}}{(n+2-j)!}\Big)$$ for each $i \geq 0$, so that \begin{IEEEeqnarray*}{rCl} \y(t,s) &=& \phi(t) + \sum_{j=1}^{n-1} s_j \phi^{(j)}(t) \\ &=& \phi(t_0) + \sum_{i=1}^{n+1} \Big [(t-t_0)^i + \sum_{j=1}^i s_j (i)_j (t-t_0)^{i-j} \Big] \big(\phi^{(i)}(t_0)/i! \big) + O \Big(C\sum_{i=1}^{n-1} |s_i|\frac{|t-t_0|^{n+2-i}}{(n+2-i)!} \Big)\\ & & + \, O\Big(C\frac{|t-t_0|^{n+2}}{(n+2)!}\Big). \IEEEyesnumber \label{8.00} \end{IEEEeqnarray*} Therefore, if we take the matrix \begin{equation} \label{phimatrix} \Phi(t) = \begin{pmatrix} \Big | & \Big | & \ & \Big | \\ \ & \ & \ & \ \\ \phi'(t) & \phi''(t)/2! & \ldots & \phi^{(n+1)}(t)/(n+1)!\\ \ & \ & \ & \ \\ \Big | & \Big | & \ & \Big | \end{pmatrix}, \end{equation} and define the affine map $\Mb: \R^{n+1} \rightarrow \R^{n+1}$ by \\$$\Mb(x) = \Phi(t_0)^{-1}\big(x - \phi(t_0)\big),$$ then \begin{equation} \label{dag} \Mb\big(\y(t,s)\big) = \x(t-t_0,s) + O \Big(C\sigma^{-1}\frac{|t-t_0|^{n+2}}{(n+2)!}\Big) + O\Big(C\sigma^{-1}\sum_{i=1}^{n-1}|s_i|\frac{|t-t_0|^{n+2-i}}{(n+2-i)!}\Big) \end{equation} where $\sigma$ is the smallest singular value of $\Phi(t_0)$. The verity of \eqref{dag} follows from the fact that linear maps take unit balls to ellipsoids, which is seen directly from singular value decomposition of matrices. Furthermore, the ellipsoid's largest principal axis length corresponds to the maximum singular value. For $\Phi^{-1}(t_0)$, this value is $\sigma^{-1}$. 

In other words, $\Mb$ maps $\Mf$ into a perturbation of $\M$. The representation \eqref{8.00} also rationalizes our definition of $\Nc_\delta(\Mf)$ in \eqref{nbhd}. Given an arbitrary point $\y(t,s) + O\big(\delta/(n+1)!\big)\phi^{(n+1)}(t) \in \Nc_\delta(\Mf)$, we can approximate $$\phi^{(n+1)}(t) = \phi^{(n+1)}(t_0) + O(C|t-t_0|)$$ and consequently $$\Mb \Big(\y(t,s) + O\big(\delta/(n+1)!\big)\phi^{(n+1)}(t) \Big) = \x(t,s) + O(\delta)\ev_{n+1} $$ \begin{equation} \label{8.01}  +\,  O\bigg(C\sigma^{-1}\Big(\frac{|t-t_0|^{n+2}}{(n+2)!} + \sum_{i=1}^{n-1} |s_i| \frac{|t-t_0|^{n+2-i}}{(n+2-i)!} \Big)\bigg) + O\bigg(C\sigma^{-1}\delta|t-t_0|/(n+1)!\bigg).\end{equation} 

Now, accessible lower bounds for $\sigma$ do exist. We take the following inequality from \cite{DY} which aligns well with emphasizing $\det(\Phi)$ and $\|\phi\|_{C^{n+2}}$\,: $$\sigma \geq |\det \Phi(t_0)| \cdot \Big(\frac{n}{\|\Phi(t_0)\|^2}\Big)^{(n-1)/2}$$ which then, assuming as we may that $\det \Phi(t)$ is bounded below uniformly by $c$, gives rise to \begin{equation} \label{8.02} \sigma \gtrsim c \cdot C^{-(n-1)}.\end{equation}\\

At last, we formally define the decoupling constant for $\Mf$. Given $\delta > 0$, let $\Pc_\delta(\Mf)$ be a partition of $\Mf$ consisting of caps $\Delta$ that are flat at scale $\delta$ in the following manner: $$ \exists \, \y(t_0,s) \in \Delta \text{ such that }$$ \begin{equation} \label{parti} \sup_{\y(t,s) \in \D} |t-t_0|^{n+1} + (n+1)_1 |s_1| |t-t_0|^n + \cdots + (n+1)_{n-1} |s_{n-1}| |t-t_0|^2 \lesssim \delta. \end{equation} The inequality is familiar and should be recognized as our sufficient criterion for almost flatness within $\M$. In the previous section, the inequality in \eqref{parti} was shown to hold for each element of $\Pc_\delta(\M)$ obtained there. Thus, \eqref{parti} defines the elements of $\Pc_\delta(\Mf)$ (essentially) as affine images of the elements in $\Pc_\delta(\M)$ under the maps $\Mb$. Notice that $\D \in \Pc_\delta(\Mf)$ is then indeed almost flat at scale $\delta$, relative to the transverse vector $\phi^{(n+1)}(t)/(n+1)!$.

We define $\Dec_p^\Mf(\delta)$ to be the smallest $K > 0$ such that there exists a $\Pc_\delta(\Mf)$ for which \begin{equation} \label{8.0} \|P_{\Nc_\delta(\Mf)} f\|_p \leq K (\sum_{\D \in \Pc_\delta(\Mf)} \|P_{\Nc_\delta(\D)} f\|_p^2)^{1/2}. \end{equation}

\begin{te} \label{constant2}
Let $B = \big(\frac{C\sigma^{-1}}{(n+1)!}\big)^{1/2} \lesssim \big(\frac{C^nc^{-1}}{(n+1)!}\big)^{1/2}$, where $C$ bounds the $C^{n+2}$ norm of $\phi$ and $c$ is a uniform lower bound on the determinant of \eqref{phimatrix}. Then, $$\nDec_p^\Mf(\delta) \lesssim \delta^{-2\e n^2} \Big(\log \frac{1}{\delta} \Big)^{\big[\log B\k^{2\e^{-1}n \log n} + n\e^{-1}(\log n)O(\log \log 1/\delta) \big](\log \frac{n+2}{n+1})^{-1}}  $$ $$\cdot \prod_{j=1}^{n-1} (\log j!\delta^{-\frac{j}{n+1}})^{O\big(\frac{\log \log \frac{1}{\delta}}{\log \frac{n+2}{n+1}}\big)}.$$

\end{te}

\begin{proof} 

The proof adapts the familiar Pramanik-Seeger induction on scales from \cite{PS} to the current setting. Here, given $\delta$, the appropriate scale to take at the inductive step is $\bd = \delta^{\frac{n+1}{n+2}}$. 

We proceed as follows. $P_{\Nc_\delta(\Mf)} f$ is supported in $\Nc_\delta(\Mf)$ and therefore is also Fourier supported in the larger neighborhood $\Nc_{\bd}(\Mf)$. Consequently, $P_{\Nc_\delta(\Mf)} f = P_{\Nc_{\bd}(\Mf)} f$, and we apply the inductive hypothesis to obtain the preliminary decoupling \begin{equation} \label{8.1} \|P_{\Nc_\delta(\Mf)} f\|_p \leq \Dec_p^\Mf(\bd) (\sum_{\bD \in \Pc_{\bd}(\Mf)} \|P_{\Nc_\delta(\bD)} f\|_p^2)^{1/2}. \end{equation} By \eqref{parti}, each $\bD$ has a point $\y(t_0, s) \in \bD$ such that \begin{equation} \label{8.2} \sup_{\y(t,s) \in \bD} |t-t_0|^{n+1} + (n+1)_1 |s_1| |t-t_0|^n + \cdots + (n+1)_{n-1} |s_{n-1}| |t-t_0|^2 \lesssim \delta^{(n+1)/(n+2)}. \end{equation} Immediately we determine that $\sup_{\y(t,s) \in \bD} |t-t_0| \lesssim \delta^{1/(n+2)}.$ By utilizing H\"older's inequality, we may assume the $t$-length of $\bD$ is instead at most $C^{-1}\sigma(n+1)! \delta^{1/(n+2)}$.

Let $\rt = t-t_0$. We have arranged for \begin{equation} \label{8.3} |\rt| \leq C^{-1}\sigma(n+1)! \delta^{1/(n+2)}.\end{equation} To each $\Nc_\delta(\bD)$, we apply the affine map $\Mb_{t_0}$ via change of variables $$\|P_{\Nc_\delta(\bD)} f\|_p = (\det \Mb) ^{-1/p'}\|P_{\Mb\big(\Nc_\delta(\bD)\big)} g\|_p, \qquad \quad (\hat{g} = \hat{f} \circ \Mb^{-1}).$$ We have characterized the points in $\Mb\big(\Nc_\delta(\bD)\big)$ in \eqref{8.01}. According to \eqref{8.2}, $$\frac{|\rt|^{n+1}}{(n+2)!} + \sum_{i=1}^{n-1} |s_i| \frac{|\rt|^{n+1-i}}{(n+2-i)!} \leq \frac{1}{(n+1)!} \sum_{i=0}^{n-1} (n+1)_i |s_i \rt|^{n+1-i} \lesssim \frac{\delta^{(n+1)/(n+2)}}{(n+1)!}.$$ From \eqref{8.3}, we then deduce that the elements of $\Mb\big(\Nc_\delta(\bD)\big)$ have the form \begin{equation} \label{8.4} \Mb\Big(\y(t,s) + O\big(\delta/(n+1)!\big) \phi^{(n+1)}(t) \Big) = \x(t,s) + O(\delta)\ev_{n+1} + O(\delta).\end{equation}

Points determined by \eqref{8.4} are not quite in $\Nc_\delta(\M)$, but nevertheless we may apply the decoupling of Theorem \ref{t0.1} (with neighborhood width $\delta$) upon inspection of the proof. The iterative argument outlayed in Section \ref{sgen} relied at each step on assessing the sets $$\Sc^{l} = \big \{(i_{l-1}+1)!|s_{i_{l-1}+1}^{(l)}|, \dots, \big((n-1)!|s_{n-1}^{(l)}| \big)^{1/(n-1-i_{l-1})}\big\}.$$ Such consideration was natural as it was the $t$-independent terms of $\x$ that determined the error formulated in Proposition \ref{pr4.1}. If small perturbations $O(\delta)$ are added to each component of $\x$, then we instead would have to address $\max \{|(i_{l-1}+j)!s_{i_{l-1}+j}|, O(\delta)\}$ in the place of $|s_{i_{l-1}+j}|.$ Yet this is immaterial as $i!|s_i| \geq \delta$ for each $i$, and so $i!|s_i^{(1)}| \geq d_1^i \delta$, giving rise to $$i!|s_i^{(l)}| = i!d_l^{i-i_{l-2}}|s_i^{(l-1)}| \geq d_l^{i-i_{l-2}} \delta$$ at each $l$-th step. In other words, $i!|s_i^{(l)}|$ always exceeds the value produced by the effect of $\Db_l \circ \cdots \circ \Db_1$ on $\delta$. We conclude that the proof of Theorem \ref{t0.1} proceeds for $\Mb\big(\Nc_\delta(\bD)\big)$ as before without any modification.

Thus we have secured \begin{equation} \label{8.5} \|P_{\Mb\big(\Nc_\delta(\bD)\big)} g\|_p \leq \Dec_p^\M(\delta) (\sum_{\D' } \|P_{\Nc_\delta(\D')  \bigcap \Mb\big(\Nc_\delta(\bD)\big)} g\|_p^2)^{1/2} \end{equation} where each $\D'$ is flat in the sense of the inequality in \eqref{parti}, and their $\delta$-neighborhoods partition $\Mb\big(\Nc_\delta(\bD)\big)$. Letting $\D = \Mb^{-1}\big(\Nc_\delta(\D')  \bigcap \Mb(\bD)\big) \subset \bD$, undoing the change of variables induces the final inequality $$ \|P_{\Nc_\delta(\Mf)} f\|_p \lesssim \big(C\sigma^{-1}/(n+1)!\big)^{1/2} \Dec_p^\Mf(\bd) \Dec_p^\M(\delta) \Big(\sum_\D \|P_{\Nc_\delta(\D)} f\|_p^2 \Big)^{1/2}.$$
The caps $\D$ comprise our partition $\Pc_\delta(\M)$, and we have ultimately proven \begin{equation} \label{8.6} \Dec_p^\Mf(\delta) \lesssim \big(C\sigma^{-1}/(n+1)!\big)^{1/2} \Dec_p^\M(\delta) \Dec_p^\Mf(\delta^{\frac{n+1}{n+2}}). \end{equation}

Recall $B = \big(\frac{C\sigma^{-1}}{(n+1)!}\big)^{1/2}$. Let $c = \frac{n+1}{n+2} < 1.$ Inequality \eqref{8.6} may be iterated, ultimately yielding $$\Dec_p^\Mf(\delta) \leq  B^N \prod_{i=0}^{N-1} \Dec_p^\M(\delta^{(\frac{n+1}{n+2})^i}) \Dec_p^\Mf(1/2)$$ with $N$ determined by $$\delta^{c^N} > \frac{1}{2}.$$ Quoting \eqref{7.12}, we bound each term $$\Dec_p^\M(\delta^{(\frac{n+1}{n+2})^i}) \leq (\delta^{-3\e/2(\frac{n+1}{n+2})^i})^{n/2}(\k\Dec_p^\P(\delta^{(\frac{n+1}{n+2})^i}))^{2\e^{-1}n \log n}\prod_{j=1}^{n-1} \log (j!\delta^{-(\frac{n+1}{n+2})^i\frac{j}{n+1}}).$$ We estimate \begin{IEEEeqnarray*}{rCll} (\k^{2\e^{-1} n \log n}B)^N &<& (\k^{2\e^{-1} n \log n}B)^{\frac{\log \frac{\log{1/4}}{\log \delta}}{\log c} } & = \Big(\frac{\log 1/4}{\log \delta}\Big)^{\frac{\log B\k^{2\e^{-1} n \log n}}{\log c} }\\  &\leq& \Big(\log \frac{1}{\delta}\Big)^{\frac{\log B\k^{2\e^{-1} n \log n}}{\log c^{-1}}}\end{IEEEeqnarray*} and also $$\prod_{i=0}^{N-1} (\delta^{-3\e/2(\frac{n+1}{n+2})^i})^{n/2}  = (\delta^{-3\e n/4})^{\sum_{i=0}^{N-1} (\frac{n+1}{n+2})^i} \leq \delta^{-2\e n^2}.$$ Employing monotonicity of the decoupling constant and implementing the bound $\Dec_p^\P(\delta) \leq (\log 1/\delta)^{O(1)}$, we deduce $$\prod_{i=0}^{N-1}\Dec_p^\P(\delta^{(\frac{n+1}{n+2})^i}) \leq \Big(\log \frac{1}{\delta}\Big)^{O\big(\frac{\log \log \frac{1}{\delta}}{\log c^{-1}}\big)}.$$ Similarly, $$\prod_{i=0}^{N-1}\prod_{j=1}^{n-1} \log(j!\delta^{-(\frac{n+1}{n+2})^i\frac{j}{n+1}}) \leq \prod_{j=1}^{n-1} (\log j!\delta^{-\frac{j}{n+1}})^{O\big(\frac{\log \log \frac{1}{\delta}}{\log c^{-1}} \big)}.$$ 

Altogether, it has been shown that $$\Dec_p^\Mf(\delta) \lesssim \delta^{-2\e n^2} \Big(\log \frac{1}{\delta} \Big)^{\big[\log B\k^{2\e^{-1}n \log n} + n\e^{-1} (\log n)O(\log \log 1/\delta) \big](\log \frac{n+2}{n+1})^{-1}}$$ $$\cdot \,\prod_{j=1}^{n-1} (\log j!\delta^{-\frac{j}{n+1}})^{O\big(\frac{\log \log \frac{1}{\delta}}{\log \frac{n+2}{n+1}}\big)}.$$ 

\end{proof}
$$ $$
\section*{Appendix: Construction of $\Pc_\delta(\M^n)$ with almost rectangular caps} 

The caps $\D$ presented for Theorems \ref{t0} and \ref{t0.1} were verified only to be almost flat, meaning that $\D$ lies within a cylinder of height $\delta$ over an adjacent tangent plane. In many typical circumstances, such geometric information is sufficient. However, there are scenarios where it may be desired for $\D$ to exhibit flatness in each of the tangent directions too. Equivalently, we might hope that $\D$ can be approximated by a rectangular box. Note that it is the lack of ``complete" flatness, i.e. presence of some principal curvature, that allows for decoupling. If a given decoupling partition does not portray almost rectilinearity, we should anticipate that it is then not maximal, and the caps may be partitioned further.

For reasons that will be divulged next, we strongly suspect that not all of the caps in our $\Pc_\delta(\M)$ derived previously are almost rectangular (although a substantial number are). Therefore, it is our goal in this appendix to address this consideration and produce a $\ell^2$ decoupling partition that is assuredly maximal.

\subsection*{Definition of Almost Rectangular Caps} In this section, we derive a precise description of almost rectangular caps within $\M$. Indeed, this is sufficient for the general case as we have affine transformations mapping small perturbations of $\M$ to any given ruled curve-generated hypersurface. It shall be clear from the definition below that such maps preserve approximation by rectangles. Formally, almost rectangular caps $\D$ have been defined in the literature as follows. $\D$ must be a subset of $\M$ subject to the condition that there is a rectangular parallelepiped (or ``box") $\Rc$ with $O(1)$ enlargement $\tilde{\Rc}$ satisfying \begin{equation} \label{6.01} \Rc \subset \Nc_\delta(\D) \subset \tilde{\Rc}. \end{equation} From the standpoint of $\ell^2$ decoupling, the finest partition of $\M$ that possibly can be taken is one where the caps are almost rectangular in a maximal sense (see Proposition in \cite{}). Almost flatness is not the limiting mark; its presence only implies that curvature, and hence decoupling, must be acquired elsewhere, specifically by analysis of the projection of $\M$ onto a tangent plane. 

Let $\{\textbf{f}_1, \dots, \textbf{f}_{n+1}\}$ be a vector basis that corresponds to the orientation of $\Rc$. Condition \eqref{6.01} implies that the coordinate lengths of $\Nc_\delta(\D)$ relative to $\{\textbf{f}_1, \dots, \textbf{f}_{n+1}\}$ are essentially the corresponding side lengths $\mathcal{L}_i$ of $\Rc$. In other words, fixing the origin to be a point in $\Rc$, the second containment in \eqref{6.01} is equivalent to the inequality \begin{equation} \label{bound} |\textbf{f}_i \cdot q| \lesssim \mathcal{L}_i \end{equation} being true uniformly for each $q \in \Nc_\delta(\D)$. Such a condition can be implied formulaically for moment surfaces using their parametric description. 

\begin{lemm} \label{lemA} Let $\D = \x([0,T] \times \prod_{i=1}^{n-1} \a_i \Ic_{k_i})$, where \begin{equation} \label{a1} T \leq (1/10)\min \big\{\min_{2 \leq i \leq n-1} i^{-1}2^{k_{i-1} - k_{i}}, (2^{k_{n-1}}\delta/(n+1)!)^{1/2}, 2^{-k_1} \big\}. \end{equation} Then, $\D$ is almost rectangular. \end{lemm}

\begin{proof}

First we comment that the hypothesis implies that $T$ also satisfies $$T \leq (1/10) 2^{-k_1}$$ since $2^{-k_{n-1}} \geq \delta^{\frac{n-1}{n+1}}$ and $2^{-k_1} \geq \delta^{\frac{1}{n+1}}.$

Recall that the tangent plane to $\M$ at any $\x(0,s)$ is the hyperplane $\xi_{n+1} = 0$ (to be hereafter identified with $\R^n$). We first observe that there is an $(n-1)$-dimensional rectangle $R \subset \D \cap \R^n$. Note that $\{t=0\} \cap \D$ is the rectangle $$R = \a_11!\Ic_{k_1} \times \cdots \times \a_{n-1}(n-1)!\Ic_{k_{n-1}} \times \{0\} \times \{0\}.$$ In turn, the hypothesis implies that the orthogonal projection $\D_{proj}$ of $\D$ onto $\R^{n-1}$ is approximated by $R$. Letting $s_i' = i!s_i$, by \eqref{a1} we have that for each $\x(t,s) \in \D$ and each $i$ \begin{IEEEeqnarray*}{rCll} t^i + \sum_{j=1}^{i-1} (i)_j |s_j| t^{i-j} &=& \sum_{j=0}^{i-1} \binom{i}{j} |s_j'| t^{i-j} \\ &\leq& \Big(\sum_{j=0}^{i-1} \binom{i}{j}\frac{1}{5^{i-j}}\frac{j!}{i!} \Big)|s_i'| & \ \leq (1/10) i!2^{-k_i} \IEEEyesnumber \label{a2}. \end{IEEEeqnarray*} We have shown that the coordinate lengths of $\D_{proj}$ relative to the orientation of $R$ correspond to the side lengths of $R$, in accordance with \eqref{bound}.

Now there isn't a rectangular box contained in $\Nc_\delta(\D)$ extending over the whole of $R$, but there is one extending over a similar $R''$ contained within. Due to \eqref{a2} (and the inequality $T \leq (1/10) |s_1|$ from \eqref{a1}) , there exist solutions $\x(T,s) \in \D$ to the systems of equations \begin{equation} \label{a3} T^i + \sum_{j=1}^i (i)_j s_j T^{i-j} = a_i, \qquad 1\leq i \leq n-1,\end{equation} so long as $|a_i| \in \mathring{\Ic}_{k_i}$ where $\mathring{\Ic}_{k_i}$ is the contraction of $\a(i)i!\Ic_{k_i}$ about its midpoint by a factor $1/2$. We take $$R' = \prod_{i=1}^{n-1} \mathring{\Ic}_{k_i} \times \{0\} \times \{0\}.$$ It remains to determine that there is a box extending over $R'$ within $\Nc_\delta(\D)$ of sufficient height.

Let $\pi: \R^{n+1} \rightarrow \R^n$ be the orthogonal projection onto the first $n$ components. The map $\pi \circ \x$ maps between Euclidean spaces of the same dimension and is continuously differentiable with derivative matrix determined by a rescaling of the columns of $\Lambda$ from \eqref{defA}. By direct computation, we see that $\pi \circ \x$ has Jacobian equal to $|s_{n-1}|\prod_{i=1}^n i!$, so that in particular it is nonzero away from $\{s_{n-1} = 0\}$ and therefore locally invertible there. Since $\D$ is the image of a connected set, it is then elementary to see that $\pi(\D)$ contains the vertical line segment connecting any point in $o_1 \in R'$ to any point $o_2 \in \pi(\D)$ directly above $o_1$. By \eqref{a2}, we know that we can take $o_2$ to have height at least $$\beta= (1/2)n!2^{-k_{n-1}}T,$$ and therefore, $\pi(\D)$ contains $\bar{R} = R' \times \a_{n-1}[0,\beta].$ Since $\beta$ is approximately the maximum height of $\pi(\D)$ over the hyperplane $\xi_n = 0$, it follows that $\bar{R}$ approximates $\pi(\D)$.

Finally, we observe that the same reasoning employed in \eqref{a2} can be used to deduce that $\D$ lies within $\delta$ of the tangent plane $\xi_{n+1} = 0$. We conclude that the rectangular box $\Rc$ of height $\delta$ over $\bar{R}$ approximates $\Nc_\delta(\D).$

\end{proof}

\begin{re} Note that the key computations \eqref{a2} and \eqref{a3} used in the proof of Lemma \ref{lemA} are certainly unaffected by permitting each $s_i$ to decrease to zero. Therefore, the proof above actually shows that $\x([0,T] \times \prod_{i=1}^{n-1} \a_i[0, 2^{1-k_i}])$ is almost rectangular also. We shall utilize this fact in the next section. \end{re}

   \subsection*{Construction of a Maximal $\ell^2$ Decoupling Partition} Lemma \ref{lemA} produces a sufficient condition for a cap to be almost rectangular. This same condition also gives us grave concern that the caps in Theorems \ref{t0} and \ref{t0.1} are not almost rectangular (though they are almost flat). We conjecture that the almost rectangular caps of $\M$, hence $\Mf$, are characterized by \eqref{a1} in Lemma \ref{lemA}, but as of yet, the author has not been able to produce a proof.
   
Nevertheless, Lemma \ref{lemA} serves as our guidepost in deriving a different and explicit decoupling partition $\overline{\Pc}_\delta(\M)$ of $\M$ comprised of almost rectangular caps. First, to draw inspiration, let us examine the final component of $\x$ from a slightly different perspective \begin{equation} \label{6.0} t^{n+1} + (n+1)_1 t^n s_1 + \cdots + (n+1)_{n-1} t^2 s_{n-1}. \end{equation} Previously, it was focus on the term $(n+1)_{n-1} t^2 s_{n-1}$, which typically dominates, that led us to record that this term is $O(\delta)$ so long as \begin{equation} \label{6.1} t \lesssim (s_{n-1}^{-1}\delta/(n+1)!)^{1/2}. \end{equation} Even when this is true, however, it does not always follow that the other terms in \eqref{6.0} are $O(\delta)$ also. The scale mentioned in \eqref{6.1} may not be sufficiently small, i.e. it may be true that \begin{equation} \label{6.2} \min\{(s_{n-1}/s_i)^{(n-1-i)^{-1}} :1 \leq i \leq n-2\} < (s_{n-1}^{-1}\delta/(n+1)!)^{1/2}. \end{equation} In this case, the value on the left side of \eqref{6.2} will be the $t$-length of maximally flat caps.

We conclude that the geometric description of flat $\D$ at most depends on two elements from $\{s_i\}$. Thus for $n \geq 5$, we are guided to the following reconfiguration of the annuli $\Ac_{k_1, \dots, k_{n-1}}.$\footnote{Note that the decoupling partition of Theorem \ref{t1} meets the condition of Lemma \ref{lemA}.} Let us label $\Ac_{k_1, \dots, k_{n-1}}$ as {\em eccentric} if it satisfies \eqref{6.2}, i.e. $$\exists \text{ } i_0 \text{ such that } (2^{-k_{n-1}+k_{i_0}})^{(n-1-i_0)^{-1}} < (2^{k_{n-1}}\delta/(n+1)!)^{1/2}.$$ Define \begin{equation} \label{6.211} A^{i,j}_{k_j, k_{n-1}} = 2^{\lceil (n-1-j)^{-1}((j-i)k_{n-1} - (n-1-i)k_j) \rceil}.\end{equation} If $(s_{n-1}/s_j)^{(n-1-j)^{-1}}$ attains the minimum in \eqref{6.2} with $s_j \sim 2^{-k_j}$ and $s_{n-1} \sim 2^{-k_{n-1}}$, then $A^{i,j}_{k_j, k_{n-1}}$ is the approximate value of $s_i$ at which \begin{equation} \label{6.21} (s_{n-1}/s_i)^{(n-1-i)^{-1}} = (s_{n-1}/s_j)^{(n-1-j)^{-1}}.\end{equation} We collect the eccentric annuli into pairwise disjoint sets of the form \begin{equation} \label{6.3} \Ac_{k_j, k_{n-1}} = \{\x(t,s) : s_j \sim 2^{-k_j}, s_{n-1} \sim 2^{-k_{n-1}}, s_i \in [0, A^{i,j}_{k_j, k_{n-1}})\}. \end{equation} The key characteristic of \eqref{6.3} is that $j$ attains the minimum in \eqref{6.2} throughout $\Ac_{k_j, k_{n-1}}$. 

Note that, in contrast to the $\Ac_{k_1, \dots, k_{n-1}}$, within $\Ac_{k_j, k_{n-1}}$ each $s_i$  extends to zero for $i \ne j$. The same will be true for our elements $\D \in \overline{\Pc}_\delta(\M^n)$ that lie within $\Ac_{k_j, k_{n-1}}$. Letting $I$ denote an interval of length $\sim n^{-1}(2^{-k_{n-1} + k_j})^{(n-1-j)^{-1}}$, such $\D$ have the form \begin{equation} \label{6.30} \x(I \times \{(s_1, \dots, s_{n-1}): s_{n-1} \in [2^{-k_{n-1}}, 2^{-k_{n-1}+1}), s_j \in [2^{-k_j}, 2^{-k_j + 1}), s_i \in [0, A^{i,j}_{k_j, k_{n-1}})\}).\end{equation} 

On the other hand, the non-eccentric $\Ac_{k_1, \dots, k_{n-1}}$ are grouped into conglomerate sets of the form \begin{equation} \label{6.4} \Ac_{k_{n-1}} = \{\x(t,s): s_{n-1} \sim 2^{-k_{n-1}}, s_i \in [0, A_{i, k_{n-1}})\} \end{equation} where $$A_{i, k_{n-1}} = \min \{2, (2^{-(n+1-i)k_{n-1}}\delta^{-(n-1-i)})^{1/2}\}.$$ As before, $A_{i, k_{n-1}}$ is the approximate value of $s_i$ at which \begin{equation} \label{6.22} (s_{n-1}/s_i)^{(n-1-i)^{-1}} = (s_{n-1}^{-1}\delta)^{1/2} \end{equation} if such equality can be attained within the compact hypersurface $\M$. Thus, letting $I$ now denote an interval of length $\sim n^{-1}(2^{k_{n-1}}\delta/(n+1)!)^{1/2}$, $\Ac_{k_{n-1}}$ will be partitioned into flat caps $\D \in \overline{\Pc}_\delta(\M)$ of the form \begin{equation} \label{6.5} \x(I \times \prod_{i < n-1} [0, A_{i, k_{n-1}}) \times [2^{-k_{n-1}}, 2^{1-k_{n-1}})). \end{equation}

Finally, we check that the caps defined in \eqref{6.30} and \eqref{6.5} are almost rectangular. It only remains to check \eqref{a1} in Lemma \ref{lemA}. 
Concerning \eqref{6.30}, we have from \eqref{6.211} that $s_i$ ranges between 0 and $$L_i \sim  ((2^{-k_{n-1}})^{i-j}(2^{-k_j})^{n-1-i})^{(n-1-j)^{-1}},$$ whose consecutive ratios satisfy $$L_i/L_{i-1} = (2^{-k_{n-1} + k_j})^{(n-1-j)^{-1}}$$ as desired. A similar computation occurs for the caps in \eqref{6.5}. 

\subsection*{$\ell^2$ Decoupling over a Maximal Partition} We have deduced a partition $\overline{\Pc}_\delta(\M)$ in the previous section that would be maximal for $\ell^2$ decoupling. It is natural to ask whether the results in this paper lead toward an $\ell^2$ decoupling over $\Pco_\delta(\M)$. Indeed, the answer is affirmative. For the sake of brevity, we only sketch the proof here. The key observation is that the decoupling inequality of Theorem \ref{t0}, also Theorem \ref{t0.1}, may be upgraded to an inequality roughly of the form \begin{equation} \label{a5} \|\sum_\t P_{\Nc_\delta(\t)} f_\t\|_p \lesssim_{\e, \delta, n}(\sum_\t \|f_\t\|_p^2)^{1/2}\end{equation} where $\{\t\}$ is an appropriately chosen refinement of $\Pc_\delta(\M)$, the functions $f_\t \in L^p(\R^{n+1})$ are arbitrary, and each $P_{\Nc_\delta(\t)} f_\t$ is the Fourier projection of $f_\t$ onto $\Nc_\delta(\t)$. Inequality \eqref{a5} may be secured using standard techniques; the approach relevant to our scenario here was pioneered in \cite{K1}. Using smooth convolution operators $\tilde{P}_\t$ to approximate $P_\t$, we may extend inequality \eqref{1.0} to the precise form of \eqref{a5} via Young's convolution inequality. In analogy with Section 7.1 of \cite{K1}, the $\tilde{P}_\t$ may be defined using the map $$\Psi(t,s,v) = \x(t,s) + v\ev_{n+1}$$ together with the rescaling maps of Section \ref{sgen} to define bump functions over each $\t$. Indeed, $\Psi$ is easily checked to be a local diffeomorphism away from $s_{n-1} = 0$, which guarantees that we may execute the construction analogous to what is done in \cite{K1}. (Note that the elements of $\Pc_\delta(\M)$ near $s_{n-1} = 0$ are derived from a lower-dimensional parabolic decoupling, and so the corresponding bump functions may be defined utilizing the geometry of a parabola in $\R^n$.) 

Once \eqref{a5} has been upheld, the argument may be concluded. Each $\t$ is contained in exactly one $\D \in \Pc_\delta(\M)$, and the number of $\t$ contained in a given $\D$ is bounded by a constant $\Cf$ that may be depend on $n$. The crucial observation to make is that each $\D$, being contained in some annulus $\Ac_{k_1, \dots, k_{n-1}} \subset \M$, was determined in \eqref{7.14} to have $t$-length $T_\D$ bounded by $$\eta = n^{3/2}2^{(-k_{n-1}+k_j)(n-1-j)^{-1}}.$$ By Lemma \ref{lemA}, we have that $$T_\D \geq n^C\eta$$ for some fixed $C \in \R$. Applying H\"older's inequality to each term on the right of \eqref{1.02}, we may assume that from the beginning $\Pc_\delta(\M)$ was chosen such that \begin{equation} \label{a20} T_\D \sim n^{-5/2 -O(1)} \eta \end{equation} uniformly in $\D$ and $k_1, \dots, k_{n-1}$. Note that the value in \eqref{a20} is essentially the $t$-length of the corresponding caps $\bD \in \Pco_\delta(\M)$ contained in $\Ac_{k_j, k_{n-1}} \supset \Ac_{k_1, \dots, k_{n-1}}$. In fact, if $I$ is the interval comprised of the $t$-values of $\bD$ and $\D^I_{k_1, \dots, k_{n-1}}$ is the cap in $\Ac_{k_1, \dots, k_{n-1}}$ with $t$-interval $I$, then $$\bD = \bigcup_{\substack{2^{-k_i} \in [0, A^{i,j}_{k_j, k_{n-1}})\\ i \notin \{j, n-1\}}} \D^I_{k_1, \dots, k_{n-1}}.$$ Choosing $f_\t = P_{\Nc_\delta(\bD)} f$ where $\bD \supset \t$ and plugging into \eqref{a5}, we obtain the following theorem.

\begin{te} \label{lastth}For each $2 \leq p \leq 6$, the following inequality holds for each fixed $\e > 0$ and all $f$: $$\|P_{\Nc_\delta(\M)} f\|_p \lesssim_n (30\delta^{-3\e/2})^{n/2}  \big(\kappa \log \big(1/\delta \big)\big)^{\e^{-1}O(n \log(n))} \Big(\sum_{\bD \in \Pco_\delta(\M)} \|P_{\bD} f\|_p^2 \Big)^{1/2}.$$ $\Pco_\delta(\M)$ is a partition of $\Nc_\delta(\M)$ by almost rectangular caps. \end{te}

Again by induction on scales, (see \cite{K1}), Theorem \ref{lastth} implies an analogous $\ell^2$ decoupling over arbitrary $\Mf$ using almost rectangular caps.

\end{document}